\RequirePackage[l2tabu, orthodox]{nag}
\documentclass[11pt]{amsart} %{{{1
\author[A. Hammerlindl]{Andy Hammerlindl}
\address{School of Mathematical Sciences, Monash University, Victoria 3800 Australia} \urladdr{ http://users.monash.edu.au/~ahammerl/}  \email{andy.hammerlindl@monash.edu}
\title{Constructing center-stable tori}

\usepackage{amsmath}
\usepackage{amssymb}
\usepackage{amsfonts}
\usepackage{amsthm}
\usepackage{amscd}
\usepackage{graphicx}
\usepackage{setspace}
\usepackage{cleveref}
\usepackage{subfigure}

\usepackage[T1]{fontenc}
\usepackage{fourier}

\makeatletter
\def\saveenum{\xdef\@savedenum{\the\c@enumi\relax}}
\def\resetenum{\global\c@enumi\@savedenum}
\makeatother

\newcommand{\ti}{\times}
\newcommand{\subof}{\subset}

\newcommand{\sans}{\setminus}
\newcommand{\slope}{\operatorname{slope}}
\newcommand{\graph}{\operatorname{graph}}

\newcommand{\bbR}{\mathbb{R}}

\newcommand{\bbZ}{\mathbb{Z}}
\newcommand{\bbN}{\mathbb{N}}
\newcommand{\bbT}{\mathbb{T}}

\newcommand{\bbTD}{\bbT^D}

\newcommand{\bbRd}{\bbR^d}

\newcommand{\UfU}{\overline{U} \sans f(U)}

\newcommand{\barU}{\overline{U}}

\newcommand{\bbS}{S^1}

\newcommand{\Es}{E^s}
\newcommand{\Ec}{E^c}
\newcommand{\Eu}{E^u}
\newcommand{\Ecu}{E^{cu}}
\newcommand{\Ecs}{E^{cs}}

\newcommand{\Ws}{W^s}

\newcommand{\Wu}{W^u}

\newcommand{\inv}{^{-1}}

\newcommand{\invn}{^{-n}}

\newcommand{\del}{\partial}

\newcommand{\interior}{\operatorname{int}}

\newcommand{\tf}{\tilde{f}}

\newcommand{\hF}{\hat{F}}
\newcommand{\hV}{\hat{V}}
\newcommand{\hL}{\hat{L}}

\newcommand{\ep}{\epsilon}
\newcommand{\lam}{\lambda}
\newcommand{\Lam}{\Lambda}

\newcommand{\al}{\alpha}

\newcommand{\bt}{\beta}

\newcommand{\gam}{\gamma}

\newcommand{\qandq}{\quad \text{and} \quad}

\newcommand{\dist}{\operatorname{dist}}

\newcommand{\id}{\operatorname{id}}

\newcommand{\Cone}{\mathcal{C}}
\newcommand{\Bone}{\mathcal{B}}
\newcommand{\Aone}{\mathcal{A}}

\newcommand{\iCone}{\interior \Cone}
\newcommand{\iBone}{\interior \Bone}
\newcommand{\iAone}{\interior \Aone}

\newcommand{\TxM}{T_x M}

\newcommand{\normalize}[1]{\frac{#1}{\|#1\|}}

\newcommand{\vertiii}[1]{{\left\vert\kern-0.25ex\left\vert\kern-0.25ex\left\vert #1 
\right\vert\kern-0.25ex\right\vert\kern-0.25ex\right\vert}}

\newenvironment{bigset}
{\bigl\{}
{\bigr\}}

\numberwithin{equation}{section}
\newtheorem{thm}[equation]{Theorem}
\newtheorem{cor}[equation]{Corollary}
\newtheorem{lemma}[equation]{Lemma}
\newtheorem{prop}[equation]{Proposition}
\newtheorem{addendum}[equation]{Addendum}

\theoremstyle{remark}

\newtheorem*{notation} {\textbf{Notation}}

\providecommand{\acknowledgement}{{\noindent \textbf{Acknowledgements}}\quad}

\begin{document}

\maketitle

%{\small
%{\sc disclaimer.}
%\ \ This paper is still in preparation.
%Please, contact me if you have any questions.
%Comments and suggestions for improvement are also most welcome.
%}
%
%\bigskip

\begin{abstract}
    We show that
    any weakly partially hyperbolic diffeomorphism on the 2-torus
    may be realized as the dynamics on a center-stable or center-unstable
    torus of a 3-dimensional strongly partially hyperbolic system.
    We also construct examples of center-stable and center-unstable
    tori in higher dimensions.
\end{abstract}

\section{Introduction} %{{{1

Partially hyperbolic dynamical systems have received a large amount of
attention in recent years.
These systems display a wide variety of highly chaotic behaviour
\cite{bon2011survey},
but they have enough structure to allow, in some cases, for the
dynamics to be understood and
classified \cite{chhu20XXpartially,hp20XXpartial}.

A diffeomorphism $f$ is \emph{strongly partially hyperbolic}
if there is a splitting
of the tangent bundle into three invariant subbundles
$TM = \Eu \oplus \Ec \oplus \Es$
such that the derivative $Df$ expands vectors in the unstable bundle $\Eu$,
contracts vectors in stable bundle $\Es$,
and these dominate any expansion or contraction in the center direction $\Ec$.
(See \cref{sec:ineq} for a precise, if slightly unorthodox, definition.)
The global properties of these systems are often determined
by analysing invariant foliations tangent to the subbundles of the splitting.

The bundles $\Eu$ and $\Es$ are uniquely integrable \cite{HPS}.
That is, there are foliations $\Wu$ and $\Ws$ such that any curve tangent to $\Eu$ or
to $\Es$ lies in a single leaf of the respective foliation.
For the center bundle $\Ec$, however,
the situation is more complicated.
There may not be a foliation tangent to $\Ec$,
and even if such a foliation exists, the bundle may not be uniquely
integrable.
The first discovered examples of partially hyperbolic systems
without center foliations were algebraic in nature.
In these examples,
both $f$ and the splitting can be taken as smooth,
and the center bundle is not integrable
because it does not satisfy Frobenius' condition of involutivity
\cite{W-thesis}.
Such non-involutive examples are only possible if the dimension of the center
bundle is at least two, and for a long time it was an open question
if a one-dimensional center bundle was necessarily integrable.

Rodriguez Hertz, Rodriguez Hertz, and Ures 
recently answered this question by
constructing a counterexample \cite{rhrhu2016coherent}.
They defined a partially hyperbolic system on the 3-torus
with a center bundle which is uniquely integrable everywhere
except for an invariant embedded 2-torus tangent to $\Ec \oplus \Eu$.
This discovery has shifted our view on the possible dynamics a
partially hyperbolic system can possess,
and leads to questions of how commonly invariant submanifolds
of this type
occur in general.
In this paper, we build further examples of partially hyperbolic systems
having compact submanifolds tangent either to $\Ec \oplus \Eu$ or
$\Ec \oplus \Es$, both in dimension 3 and in higher dimensions.

\medskip

In the construction in \cite{rhrhu2016coherent},
the dynamics on the 2-torus tangent to $\Ec \oplus \Eu$ is Anosov.
In fact, it is given by a hyperbolic linear map on $\bbT^2$, the cat map.
It has long been known that a
\emph{weakly partially hyperbolic} system, 
that is, a diffeomorphism $g : \bbT^2 \to \bbT^2$
with a splitting of the form $\Ec \oplus \Eu$
or $\Ec \oplus \Es$,
need not be Anosov.
Therefore, one can ask exactly which weakly partially hyperbolic systems may
be realized as the dynamics on an invariant 2-torus sitting inside a
3-dimensional strongly partially hyperbolic system.
We show, in fact, that there are no obstructions on the choice of dynamics.

\begin{thm} \label{thm:everycs}
    For any weakly partially hyperbolic diffeomorphism
    $g_0:\bbT^2 \to \bbT^2$,
    there is an embedding $i:\bbT^2 \to \bbT^3$
    and a strongly partially hyperbolic diffeomorphism
    $f:\bbT^3 \to \bbT^3$
    such that $i(\bbT^2)$ is either a center-stable or 
    center-unstable torus
    (depending on the splitting of $g_0)$
    and $i \inv \circ f \circ i = g_0$.
\end{thm}
To be precise, a \emph{center-stable torus} is an embedded 
copy of $\bbTD$ with $D  \ge  2$
which is tangent
to $\Ecs_f := \Ec_f \oplus \Es_f$.
Similarly, a \emph{center-unstable torus}
is tangent to $\Ecu_f := \Ec_f \oplus \Eu_f$.
We also use the terms $cs$-torus and $cu$-torus as shorthand.

In the case where the derivative of $g_0$ preserves the orientation
of the center bundle, $\Ec_{g_0}$,
we may be more specific about the construction.

\begin{thm} \label{thm:dacs}
    Let $g_0:\bbT^2 \to \bbT^2$ be a weakly partially hyperbolic
    diffeomorphism which preserves the orientation of its center 
    bundle and is homotopic to a linear Anosov diffeomorphism
    $A:\bbT^2 \to \bbT^2$ and let $0 < \ep < \tfrac{1}{2}$.
    Then there is a strongly partially hyperbolic diffeomorphism
    $f:\bbT^3 \to \bbT^3$ such that
    \begin{enumerate}
        \item $f(x,t) = (A(x), t)$ for all
        $(x,t)  \in  \bbT^3$ with $|t| > \ep$,

        \item $f(x,t) = (g_0(x), t)$ for all
        $(x,t) \in \bbT^3$ with $|t| < \tfrac{\ep}{2}$, and

        \item $\bbT^2 \times 0$ is either a center-stable or center-unstable
        torus, depending on the splitting for $g_0$.
          \end{enumerate}  \end{thm}
Since the construction is local in nature,
different weakly partially hyperbolic diffeomorphisms may be inserted into the
system at different places.

\begin{cor} \label{cor:manytori}
    Suppose $A:\bbT^2 \to \bbT^2$ is a hyperbolic linear automorphism
    and
    $g_i:\bbT^2 \to \bbT^2$
    for each $i  \in  \{1, \ldots , n\}$
    is a weakly partially hyperbolic diffeomorphism which is
    homotopic to $A$ and preserves the orientation
    of its center bundle.
    Let $\{t_1, \ldots , t_n\}$ be a finite subset of the circle, $\bbS$.
    Then there is a strongly partially hyperbolic diffeomorphism
    $f:\bbT^3 \to \bbT^3$ such that
    \begin{enumerate}
        \item $f(x,t_i) = (g_i(x), t_i)$ for 
        each $t_i$ and all $x  \in  \bbT^2$, and

        \item each
        $\bbT^2 \times t_i$ is either a center-stable or center-unstable
        torus, depending on the splitting for $g_i$.
    \end{enumerate}  \end{cor}
We also note that
the presence of a $cs$ or $cu$-torus affects the dynamics
only in a neighbourhood of that torus and does not place
global restrictions on the dynamics on $\bbT^3$.
For instance,
one could easily construct a system which has a $cs$ or $cu$-torus
$\bbT^2 \ti 0$ and has a robustly transitive blender
elsewhere on $\bbT^3$
\cite{bd1996persistent}.

The results as stated above rely on work announced by Gourmelon and Potrie
which shows that in the $C^1$-open set of diffeomorphisms of $\bbT^2$
with dominated splittings,
the subset of diffeomorphisms isotopic to a given hyperbolic toral
automorphism is connected.
See \cref{sec:useful} for further details about this property.

\medskip{}

The original construction of 
Rodriguez Hertz, Rodriguez Hertz, and Ures
on the 3-torus may be viewed as a skew product
with Anosov dynamics in the fibers.
In fact, the example can be given as a map of the form 
\[
    F(x,v) = (f(x), Av + h(x))
\]
where $f$ is a Morse-Smale diffeomorphism of the circle,
$A$ is the cat map on $\bbT^2$, and $h : S^1 \to \bbT^2$ is smooth.
The diffeomorphism $f$ has a sink at a point $x_0$ and the fiber $x_0 \ti \bbT^2$ over
this sink gives the embedded 2-torus tangent to $\Ecu_F$.

This description of $F$ naturally suggests a way to construct
higher-di\-men\-sion\-al examples of the same form.
We will show that, starting from any diffeomorphism $f$ of any closed manifold $M$,
one may construct a strongly partially hyperbolic diffeomorphism $F$ of $M \ti \bbTD$
using sinks of $f$ to construct center-unstable tori for $F$
and sources to construct center-stable tori.

\begin{thm} \label{thm:allsinkssimple}
    Let $f_0 : M \to M$ be a diffeomorphism
    and $X \subof M$ a finite invariant set such that
    every $x  \in  X$ is either a periodic source or sink.
    Then there is a strongly partially hyperbolic
    diffeomorphism $F : M \ti \bbTD \to M \ti \bbTD$
    of the form 
    \[
        F(x, v) = (f(x), Av + h(x))
    \]
    such that $f$ is isotopic to $f_0$ and, for each $x  \in  X$,
    the submanifold $x \ti \bbTD$ is tangent either to $\Ecs$ or $\Ecu$.
\end{thm}

In dimension 3, the presence of a compact submanifold tangent to
$\Ecs$ or $\Ecu$ has strong
consequences on the global topology of the manifold.
In fact,
Rodriguez Hertz, Rodriguez Hertz, and Ures 
showed that the 3-manifold can only be one of a few possibilities
\cite{RHRHU-tori}.
The proof of \cref{thm:allsinkssimple}
uses a local argument
and the global topology of $M$ has no impact on the construction.
This suggests that in higher dimensions, compact
submanifolds tangent to $\Ecs$ and $\Ecu$ may arise naturally in many
examples of partially hyperbolic systems.

\Cref{thm:allsinkssimple} is stated for the trivial fiber bundle
$M \ti \bbTD$ only for the sake of simplicity.
As the proof is entirely local in nature,
the same technique may be used to introduce
center-stable and center-unstable tori in a system defined
on a non-trivial fiber bundle, so long as the dynamics in the fibers
is given by a linear Anosov map.
By adapting the examples in \cite{goh2015partially},
it might be possible to define a system with
a center-stable torus so that the total space is simply connected.
See also \cite{gog20XXsurgery} for further constructions,
and \cite{fg2016bundles} for conditions which imply that the fiber bundle
must be trivial.
We suspect that,
just as in the case of dimension 3,
the future study of compact center-stable submanifolds in higher dimensions
will be full of surprises.

\medskip

In order to prove the theorems listed above,
dominated splittings must be constructed in a variety of settings.
Before constructing specific examples, sections \ref{sec:ineq}
and \ref{sec:follow} first introduce a number of helpful tools in a general
setting
which give sufficient and easy-to-verify conditions for dominated splittings to
exist.
\Cref{sec:useful} establishes properties for dominated splittings in dimension 2
specifically.
\Cref{sec:proofdacs} gives the proof of \cref{thm:dacs}.
\Cref{sec:further} generalizes this construction and proves \cref{thm:everycs}.
Finally, \cref{sec:higher} handles higher-dimensional examples and proves
\cref{thm:allsinkssimple}.

\section{Splittings and inequalities} \label{sec:ineq} %{{{1

Many concepts in dynamical systems are defined by an invariant splitting
with one or more inequalities related to the splitting.
This section shows that, in many cases,
the inequalities need only be verified on the non-wandering set
of the system.
The results in this section are similar
in nature to those established in \cite{cao2003exponents}
and earlier work referenced therein.
As the proofs are short, we give them here for completeness.

Throughout this section assume $f$ is a homeomorphism of a compact metric space $M$.
Let $NW(f)$ denote its non-wandering set.

\begin{prop} \label{prop:outsideNW}
    If $U$ is a neighborhood of $NW(f)$,
    there is a uniform bound $N$ such that
    any orbit
    \begin{math}
        \{f^n(x) : n  \in  \bbZ \}  \end{math}
    has at most $N$ points lying outside of $U$.
\end{prop}
\begin{proof}
    Suppose no such $N$ exists.
    As $M \sans U$ is totally bounded,
    for any $k  \in  \bbN$,
    there is a point $x_k  \in  M \sans U$ and
    an iterate $n_k  \ge  1$
    such that
    $d(x_k$, $f^{n_k}(x_k)) < \tfrac{1}{k}$.
    The sequence $\{x_k\}$ accumulates on a non-wandering point outside
    of $U$, which gives a contradiction.
\end{proof}

A \emph{cochain} for $f$ (in the context of this section) is a collection of
continuous functions
$\al_n:X \to \bbR$ for $n  \in  \bbZ$.
The cochain is \emph{additive} if
\[
    \al_{n+m}(x) = \al_n(f^m(x)) + \al_m(x)
\]
for all $n,m  \in  \bbZ$ and $x  \in  X$.
It is \emph{superadditive} if
\[
    \al_{n+m}(x)  \ge  \al_n(f^m(x)) + \al_m(x)
\]
for all $n,m  \in  \bbZ$ and $x  \in  X$.
It is \emph{eventually positive} if there is $n_0$ such that
$\al_n$ is positive for all $n > n_0$.
Note that any positive linear combination of superadditive cochains
is again superadditive.

\begin{prop} \label{prop:superadd}
    If $\al$ is a superadditive cochain, the following are equivalent{:}
    \begin{enumerate}
        \item $\al$ is eventually positive;

        \item there is $n  \ge  1$ such that
        $\al_n(x) > 0$ for all $x  \in  M$;

        \item there is $n  \ge  1$ such that
        $\al_n(x) > 0$ for all $x  \in  NW(f)$;
    \end{enumerate}  \end{prop}
\begin{proof}
    Clearly (1) implies (2) and (2) implies (3).

    Proof of (2) implies (1){:}
    Suppose (2) holds for some $n$.
    As $\al_n$ and $\al_1$ are continuous,
    there are $\delta > 0$ and $C > 0$ such that
    $\al_n(x) > \delta$ and
    $\al_1(x) > -C$ for all $x  \in  M$.
    Write $m  \in  \bbZ$ as $m = q n + r$ with
    $q  \in  \bbZ$ and $0  \le  r < n$.
    Then
    \begin{math}
        \al_m(x)  \ge  q \delta - C n.
    \end{math}
    If $m$ is sufficiently large and positive, then so is $q \delta - C n$. 

    Proof of (3) implies (2){:}
    First, note that if $\al$ is a superadditive cochain
    for $f$ and $k  \ge  1$,
    then $\bt_n = \al_{nk}$ defines a superadditive cochain for $f^k$.
    Therefore,
    we may assume $\al_1(x) > 0$ for all $x  \in  NW(f)$.
    Next, if $\gam$ is the unique additive cochain with $\gam_1 = \al_1$,
    then $\al_n  \ge  \gam_n$ for all $n  \ge  1$.
    Therefore,
    we may assume $\al$ is additive.
    Let $\ep > 0$ be small enough that
    \begin{math}
        U := \{ x  \in  M : \al_1(x) > \ep \}
    \end{math}
    is a neighborhood of $NW(f)$.
    Let $N$ be the bound in \cref{prop:outsideNW},
    and
    let $C$ be such that $\al_1(x) > -C$ for all $x  \in  M$.
    Then $\al_m(x) > \ep(m-N) - C N$
    for all $m$ and $x$.
    Thus, for large $m$, $\al_m$ is positive.
\end{proof}
For a linear operator, $A$, between normed vector spaces,
the norm $\|A\|$ and conorm $m(A)$ are defined by
\[
    \|A\| = \sup \{\|A v\|: \|v\| = 1 \}
    \qandq
    m(A) = \inf \{\|A v\|: \|v\| = 1 \}.
\]
If $f$ is a diffeomorphism and $E \subof TM$ is
a continuous invariant subbundle,
then each of
$\log m(Df^n|_{E(x)})$
and
$-\log \|Df^n_{E(x)}\|$
defines a superadditive cochain.
We formulate a number of dynamical concepts
in terms of linear combinations of such cochains.
Here, all bundles considered are non-zero.
\begin{enumerate}
    \item An invariant subbundle $E$ is \emph{expanding} if
    \[        \log m(Df^n|_{E(x)})  \]
    is eventually positive.

    \item An invariant subbundle $E$ is \emph{contracting} if
    \[        -\log \|Df^n_{E(x)}\|  \]
    is eventually positive.

    \item An invariant splitting $\Eu \oplus \Es$ is \emph{dominated} if
    \[        \log m(Df^n|_{\Eu(x)}) - \log \|Df^n|_{\Es(x)}\|  \]
    is eventually positive.
    Write $\Eu \oplus_> \Es$ to indicate the direction
    of the domination.

    \item An invariant splitting $\Eu \oplus \Es$ is \emph{absolutely dominated} if
    there is a constant $c  \in  \bbR$ such that both
    \[        \log m(Df^n|_{\Eu(x)}) - c n
        \qandq
        c n - \log \|Df^n|_{\Es(x)}\|  \]
    are eventually positive.

    \item A dominated splitting
    $\Eu \oplus_> \Es$
    is \emph{hyperbolic} if
    $\Es$ is contracting
    and
    $\Eu$ is expanding.

    \item A dominated splitting is
    \emph{weakly partially hyperbolic} if it is either
    of the form
    $\Ec \oplus_> \Es$ with $\Es$ contracting
    or
    $\Eu \oplus_> \Ec$ with $\Eu$ expanding.

    \item An invariant splitting
    $\Eu \oplus \Ec \oplus \Es$
    is 
    \emph{strongly partially hyperbolic} if both
    $(\Eu \oplus \Ec) \oplus_> \Es$ and
    $\Eu \oplus_> (\Ec \oplus \Es)$ are dominated splittings,
    $\Es$ is contracting,
    and
    $\Eu$ is expanding.

    \item For $r  \ge  1$,
    a strongly partially hyperbolic splitting is
    $r$-\emph{partially hyperbolic} if
    both
    \[        \log m(Df^n|_{\Eu(x)}) - r \log \|Df^n|_{\Ec(x)}\|  \]
    and
    \[        r \log m(Df^n|_{\Ec(x)}) - \log \|Df^n|_{\Es(x)}\|  \]
    are eventually positive.\\
    Sometimes,
    one also requires that $f$ is a $C^r$ diffeomorphism
    \cite{HPS}.

    \item A strongly partially hyperbolic splitting is
    \emph{center bunched}
    if both
    \[        \log m(Df^n|_{\Eu(x)})
        - \log \|Df^n|_{\Ec(x)}\|
        + \log m(Df^n|_{\Ec(x)})  \]
    and 
    \[        - \log \|Df^n|_{\Ec(x)}\|
        + \log m(Df^n|_{\Ec(x)})
        - \log \|Df^n|_{\Es(x)}\|  \]
    are eventually positive.
\end{enumerate}
\begin{cor} \label{cor:nwineq}
    Let $f$ be a diffeomorphism on a compact manifold.
    For an invariant splitting,
    any of the properties listed above holds on all of $M$
    if and only if the property holds on the non-wandering set.
\end{cor}
Since the log of the Jacobian of $Df^n|_{E(x)}$
defines an additive cochain,
one could also establish similar results for volume partial hyperbolicity
as studied in \cite{bdp2003dichotomy}.
Further, the techniques in \cite{cao2003exponents}
show that all of these properties hold uniformly if and only if
they hold in a non-uniform sense on all invariant measures.

\section{Splittings from sequences} \label{sec:follow} %{{{1

Here we present
what are hopefully ``user-friendly''
techniques to prove the existence
of a dominated splitting.
The techniques here are similar in form to results developed
by Ma\~n\'e to study quasi-Anosov systems
\cite[Lemma 1.9]{mane1977quasi},
by Hirsch, Pugh, Shub in regards to normally hyperbolicity
\cite[Theorem 2.17]{HPS},
and by Franks and Williams in constructing non-transitive Anosov flows
\cite[Theorem 1.2]{fw1980anomalous}.

This section uses $\Eu$ and $\Es$ to denote the bundles of a dominated
splitting, even though the splitting may not necessarily be
uniformly hyperbolic.
It is far easier, at least for the author,
to remember that $\Eu$ dominates $\Es$
than
to remember which of, say, $E^1$ and $E^2$
dominates the other.

\begin{notation}
    For a non-zero vector $v  \in  TM$
    and $n  \in  \bbZ$,
    let $v^n$ denote the unit vector
    \[
        v^n = \normalize{Df^n v}.
    \]
    Of course, $v^n$ depends on the diffeomorphism $f:M \to M$
    being studied, so this notation is used only when the
    $f$ under study is clear.
\end{notation}
\begin{thm} \label{thm:chaindom}
    Suppose $f$ is a diffeomorphism
    of a closed manifold $M$ and
    $Z$ is an invariant subset which contains all
    chain-recurrent points and has
    a dominated splitting
    \[
        T_Z M = \Eu \oplus \Es
          \]
    with $d = \dim \Eu$.
    Suppose that for every $x  \in  M \sans Z$,
    there is a point $y$ in the orbit of $x$
    and a subspace $V_y$ of dimension $d$ such that
    for any non-zero $v  \in  V_y$, each of the sequences
    $v^{n}$ and $v^{-n}$ accumulates on $\Eu$ as $n \to \infty$.
    Then,
    the dominated splitting on $Z$
    extends to a dominated splitting on all of $M$.
\end{thm}
%In the previous section, the non-wandering set
%was the invariant subset of importance.
%Here, it is the chain-recurrent set. 
%It is not clear
%if \cref{thm:chaindom} holds under the weaker
%condition that $NW(f) \subof Z$.

A key step in proving the theorem is the following

\begin{prop} \label{prop:attractorsplit}
    Let $f:M \to M$ be a diffeomorphism,
    $\Lam$ a compact invariant subset,
    and let $U \subof \Lam$ be open in the topology of $\Lam$
    such that
    \begin{enumerate}
        \item $f(U)$ is compactly contained in $U$,

        \item each of
        \[            \bigcap_{n>0} f^n(U)
            \qandq
            \bigcap_{n>0} \Lam \sans f \invn(U)  \]
        has a dominated splitting with
        \begin{math}
            d = \dim \Eu,  \end{math}
        and

        \item for each $x  \in  \UfU$
        there is a $d$-dimensional subspace $V_x$
        such that for all $0  \ne  v  \in  V_x$,
        both
        $v^n$ and $v^{-n}$
        accumulate on $\Eu$ as $n \to \infty.$
    \end{enumerate}
    Then,
    there is a dominated splitting on all of $\Lam$.
\end{prop}
From the proof,
it will be evident that 
if $x  \in  \overline{U} \sans f(U)$,
then $\Eu(x) = V_x$
in the resulting dominated splitting on $M$.
Therefore, it is not immediately clear
how applying
\cref{thm:chaindom} or \cref{prop:attractorsplit}
would compare favorably
to constructing a dominated splitting directly.
Still,
there are a number of advantages.
First,
only $\Eu$ needs to be known, not $\Es$,
and only on a single fundamental domain
where, depending on $f$, it may be easy to define.
Next,
to verify the hypotheses,
one need only consider individual convergent subsequences
rather than an entire cone field or splitting.
Finally, 
as long as one already knows that the original splitting on $Z$ is
dominated,
there are no further inequalities to verify.

While cone fields do not appear in the statement of
\cref{prop:attractorsplit},
they are needed for its proof.
We follow the conventions given in
\cite[Section 2]{cropot2015lecture}.
If $\Lam \subof M$ and $\Cone \subof T_\Lam M$ is a cone field,
then for each $x  \in  \Lam$,
the cone $\Cone(x)$ at $x$ is of the form
\[
    \Cone(x) = \{ v  \in  T_x M : Q_x(v)  \ge  0 \}.
\]
where $Q_x$ is a non-positive, non-zero quadratic form
which depends continuously on $x  \in  \Lam$.
The \emph{interior} of $\Cone(x)$ is
\[
    \iCone(x) := \{0\} \cup \{ v  \in  T_x M : Q_x(v) > 0 \}
\]
and the \emph{dual cone} is
\[
    \Cone^*(x) := \{ v  \in  T_x M : \, -Q_x(v)  \ge  0 \}.
\]
\begin{lemma} \label{lemma:coneprops}
    Let $\Lam \subof M$ be an invariant set with
    a dominated splitting
    \begin{math}
        T_\Lam M = \Eu \oplus \Es.
    \end{math}
    Then there is a neighborhood $U$ of $\Lam$ and a cone field
    $\Cone$ defined on $U$ 
    such that
    \begin{enumerate}
        \item if a sequence $\{v_k\}$ of unit vectors in TM
        converges to $v  \in  \Eu,\\$
        then $v_k  \in  \Cone$ for all large positive $k$;

        \item if $x$, $f(x)  \in  U$,
        then
        $Df(\Cone(x)) \subof \iCone(f(x))$;

        \item if $x  \in  M$ and $N  \in  \bbZ$ are such that $f^{-n}(x)  \in  U$ for all $n > N$,
        then
        \[            \bigcap_{n > N} Df^n \bigl( \Cone(f^{-n}(x)) \bigr)  \]
        is a subspace of $\TxM$ with
        the same dimension as $\Eu$;

        \item if $x  \in  M$ and $N  \in  \bbZ$ are such that $f^{n}(x)  \in  U$ for all $n > N$,
        then
        \[            \bigcap_{n>N} Df^{-n} \bigl(\Cone^*(f^n(x)) \bigr)  \]
        is a subspace of $\TxM$ with the same dimension as $\Es$.

        \item the subspaces given by (3) and (4) define an extension
        of the dominated
        splitting to all of \, 
        \begin{math}
            \bigcap_{n  \in  \bbZ} f^n(U).
        \end{math}  \end{enumerate}  \end{lemma}
The proof of \cref{lemma:coneprops}
uses the same techniques as in
\cite[Section 2]{cropot2015lecture}
and is left to the reader.

\begin{lemma} \label{lemma:boneincone}
    In the setting of \cref{prop:attractorsplit},
    if there are cone fields
    $\Bone$ defined on $\Lam \sans f(U)$
    and
    $\Cone$ defined on $\overline{U}$
    such that
    $d = \dim \Bone = \dim \Cone$
    and
    \begin{align*}
        Df(\Bone(x)) &\subof \iBone(f(x))
        &\text{ if } x  \in  M \sans U,\\
        Df(\Cone(x)) &\subof \iCone(f(x))
        &\text{ if } x  \in  \overline{f(U)},\\
        \Bone(x) &\subof \Cone(x)
        &\text{ if } x  \in  \overline{U} \sans f(U),
    \end{align*}
    then there is a dominated splitting of dimension $d$
    defined on all of $\Lam$.
\end{lemma}
\begin{proof}
    Let $\al:\Lam \to [0,1]$ be a continuous function such that
    $\al(M \sans U) = \{0\}$ and $\al(f(U)) = \{1\}$.
    If $P_x$ is the continuous family of quadratic forms defining $\Bone$
    and $Q_x$ is the family defining $\Cone$,
    then
    \[
        (1 - \al(x)) P_x + \al(x) Q_x
    \]
    defines a cone field $\Aone$ on $\Lam$
    such that $Df(\Aone(x)) \subof \iAone(f(x))$
    for all $x  \in  \Lam$.
    This inclusion implies the existence of a dominated splitting.
\end{proof}
\begin{proof}
    [Proof of \cref{prop:attractorsplit}]
    Let $\Lam_C$ and $\Lam_B$ denote the two intersections respectively
    in item (2) of the proposition.
    By \cref{lemma:coneprops},
    there is a cone field $\Cone_0$
    defined on a neighborhood $U_C$ of $\Lam_C$.
    For $n  \in  \bbZ$,
    define a cone field
    $\Cone_n$
    on $f^{n}(U_C)$
    by
    \begin{math}
        \Cone_n(x) = Df^n(\Cone_0(f^{-n}(x))).
    \end{math}
    Similarly,
    define a cone field $\Bone_0$ on a neighborhood $U_B$ of $\Lam_B$
    and for each $n  \in  \bbZ$
    define the cone field
    \begin{math}
        \Bone_n(x) = Df^n(\Bone_0(f^{-n}(x))).
    \end{math}
    \smallskip{}

    We claim here that
    \begin{math}
        \bigcap_m \Bone_m(x) = V_x  \end{math}
    for all $x  \in  \UfU$
    where the intersection is over all $m  \in  \bbZ$
    for which $\Bone_m(x)$ is defined
    and $V_x$ is the subspace given in the statement of the proposition.
    Indeed, if $v  \in  V_x$ is non-zero, then there is a sequence
    $n_j \to \infty$
    such that
    $v^{-n_j}$ converges to a vector in $\Eu$.
    Hence,
    $v^{-n_j}  \in  \Bone_0$ for all large $j$.
    Equivalently,
    $v  \in  \Bone_{n_j}$ for all large $j$.
    Since the sequence $\Bone_n$ is nested,
    \[    
        \bigcap_{j} \Bone_{n_j}(x) = \bigcap_n \Bone_n(x).
    \]
    This shows that $V_x \subof \bigcap_n \Bone_n(x)$.
    Since both sets are $d$-dimensional subspaces
    of $T_x M$, they must be equal.
    This proves the claim

    \smallskip{}

    If, for some $m,n  \in  \bbZ$, the cone fields
    $\Bone_m$ and $\Cone_n$
    satisfied the conditions of \cref{lemma:boneincone},
    then the desired dominated splitting would exist.
    Hence,
    we may assume that for every $m$, $n  \in  \bbZ$,
    the open set
    \[
        \{x : \Bone_m(x) \subof \interior \Cone_n(x) \}
    \]
    does not cover all of $\UfU$.
    By compactness,
    there is $y  \in  \UfU$
    such that
    \[
        \Bone_m(y) \sans \interior \Cone_n(y)
    \]
    is non-empty for all $m$, $n  \in  \bbZ$.
    By compactness of the unit sphere in $T_y M$,
    the intersection
    \[
        \bigcap_{m,n} \Bone_m(y) \sans \interior \Cone_n(y)
    \]
    is non-empty.
    Let $u$ be a unit vector in this intersection.
    Since $u  \in  \bigcap_m \Bone_m(y)$,
    the above claim shows that
    $u  \in  V_y$.
    Therefore,
    there is $n_j \to \infty$ such that $u^{n_j}$
    converges to a vector in $\Eu$.
    Then $u^{n_j}  \in  \Cone_0$ for all large $j$,
    and therefore
    \begin{math}
        u  \in  \Cone_{-n_j} \subof \interior \Cone_{-n_j-1}  \end{math}
    for all large $j$ as well.
    This gives a contradiction.
\end{proof}
\medskip{}

%When no ambiguity arises,
%we denote the subspaces in the last two items above as
%$\Eu(x)$ and $\Es(x)$ respectively.

\begin{proof}
    [Proof of \cref{thm:chaindom}]
    By the so-called ``Fundamental Theorem of Dynamical Systems''
    due to Conley
    \cite{nor1995fundamental},
    there is a continuous function $\ell:M \to \bbR$
    such that $\ell(f(x))  \le  \ell(x)$
    with equality if and only if $x$ is in the set $R(f)$
    of chain-recurrent points.
    Further,
    $\ell(R(f)$) is a compact, nowhere dense subset of $\bbR$.

    Let $\Cone$ be a cone field defined on a 
    neighborhood $U$ of $R(f)$ as in
    \cref{lemma:coneprops}.
    Then, there is $\delta > 0$ such that
    \begin{math}
        \ell(x) - \ell(f(x)) > \delta
    \end{math}
    for all $x \notin U$.
    Define
    \[
        a_1 < b_1 < a_2 < b_2 < \cdots < a_q < b_q
    \]
    such that $b_i - a_i < \delta$ for all $i$
    and the union of closed intervals $[a_i,b_i]$
    covers $\ell(R(f)$).
    For $a$$,b  \in  \bbR$ define
    \[
        \Lam[a,b] :=
            \begin{bigset}
            x  \in  M : \ell(f^n(x))  \in  [a,b]
            \text{ for all }
            n  \in  \bbZ  \end{bigset}
    \]
    If $x  \in  \Lam[a_i,b_i]$,
    then $b_i - a_i < \delta$
    implies that $f^n(x)  \in  U$ for all $n$.
    Therefore, the dominated splitting may be extended to each $\Lam[a_i,b_i]$.
    By the inductive hypothesis,
    assume the dominated splitting
    has been extended to all of $\Lam[a_1,b_k]$.
    Choose $t_k  \in  (b_k,a_{k+1})$ and
    use
    \[    
        \Lam = \Lam[a_i,b_{k+1}]
        \qandq
        U = \{ x  \in  \Lam : \ell(x) < t_k \}
    \]
    in \cref{prop:attractorsplit}
    to extend the dominated splitting
    to all of $\Lam$.
    By induction,
    the dominated splitting extends to all of $\Lam[a_1,b_q] = M$.
\end{proof}
When applying \cref{thm:chaindom},
it may be a hassle
to show directly that
$v^n$ accumulates on $\Eu$.
Suppose instead we know that there is a sequence
$\{n_j\}$
with $\lim_j n_j = +\infty$ such that
$v^{n_j}$ converges to a unit vector
$w  \in  T_Z M$ which does not lie in $\Es$.
As with $v$, we use the notation
\[
    w^m = \normalize{Df^m(w)}.
\]
The properties of the dominated splitting on $Z$
imply that there is a sequence $\{m_j\}$
tending to $+\infty$
such that $\lim_j w^{m_j}$ exists and lies in $\Eu$.
By replacing $\{n_j\}$ with a further subsequence,
one may establish that $\lim_j v^{n_j + m_j} = \lim_j w^{m_j}$.
This reasoning shows that
if $v^{n}$ accumulates on a vector in $T_Z M \sans \Es$,
it also accumulates on a vector in $\Eu$.

Iterating in the opposite direction,
suppose there is a sequence $\{n_j\}$ tending to $+\infty$
such that $\{v^{-n_j}\}$ converges to $w  \in  T_Z M \sans \Es$.
Then there is a sequence $\{m_j\}$ tending to $+\infty$
such that $\lim_j w^{m_j}$ exists and lies in $\Eu$.
By replacing $\{n_j\}$ with a subsequence,
one may establish both that $\lim_j (-n_j + m_j) = -\infty$
and $\lim_j v^{-n_j + m_j} = \lim_j w^{m_j}$.
Hence,
if $v^{-n}$ accumulates on a vector in $T_Z M \sans \Es$,
it also accumulates on a vector in $\Eu$.

With these observations in mind,
we now state a slightly generalized version of \cref{thm:chaindom}.
The proof is highly similar
and is left to the reader.

\begin{thm} \label{thm:genchaindom}
    Suppose $f$ is a diffeomorphism of a manifold $M$,
    and $Y$ and $Z$ are compact invariant subsets such that
    \begin{enumerate}
        \item all chain recurrent points of $f|_Y$ lie in $Z$,

        \item $Z$ has a dominated splitting $T_Z M = \Eu \oplus \Es$
        with $d = \dim \Eu$,
        and

        \item for every $x  \in  Y \sans Z$,
        there is a point $y$ in the orbit of $x$
        and a subspace $V_y$ of dimension $d$ such that
        for any non-zero $v  \in  V_y$, each of the sequences
        $v^{n}$ and $v^{-n}$ accumulates on a vector
        in $T_Z M \sans \Es$ as $n \to +\infty$.
    \end{enumerate}
    Then the dominated splitting on $Z$ extends to a dominated splitting
    on $Y \cup Z$.
\end{thm}
\section{Splittings on the 2-torus} \label{sec:useful} %{{{1

This section introduces a number of properties of dominated splittings
in dimension 2
that will be used in the next section to prove \cref{thm:dacs}.
First, we state the announced result of
Gourmelon and Potrie mentioned in the introduction.

\begin{prop} \label{prop:isodom}
    If $g_0:\bbT^2 \to \bbT^2$ has a global dominated splitting
    and
    $g_0$ is isotopic to a linear Anosov diffeomorphism
    $A:\bbT^2 \to \bbT^2$,
    then there is a continuous parameterized family of diffeomorphisms
    $g: \bbT^2 \times [0,1] \to \bbT^2$
    such that
    $g(\cdot,0) = g_0$,
    $g(\cdot,1) = A$,
    and each $g(\cdot,t)$ for $t \in [0,1]$
    has a dominated splitting.
\end{prop}
If, out of caution, one wants to avoid using this announced but not yet
published result, then a condition must be added to the $g_0$ and $g_i$
in theorems \ref{thm:everycs} and \ref{thm:dacs}
that they lie in the connected component of $A$.
It is already well established that this connected component contains
weakly partially hyperbolic diffeomorphisms which are not Anosov.
For completeness, \cref{prop:bifurcate} below gives a specific example of a weakly
partially hyperbolic system which can be reached from a linear Anosov system
by a path in the space of systems with dominated splittings.

\begin{addendum}
    We may assume $g$ in \cref{prop:isodom} is a $C^1$ function both in $\bbT^2$
    and $[0,1]$.
\end{addendum}
\begin{proof}
    [Proof of addendum]
    Suppose that $g$ is $C^0$ in the parameter $[0,1]$.
    As diffeomorphisms with a dominated splitting comprise an open subset
    of all $C^1$ diffeomorphisms,
    there is a partition
    $0 = t_0 < t_1 < \cdots < t_m = 1$
    such that if $t \in [t_i, t_{i+1}]$
    then the linear interpolation defined by
    \[
        x \mapsto g(x, t_i) +
        \frac
            {t - t_i}
            {t_{i+1} - t_i}
        [g(x,t_{i+1} - g(x,t_i)]
    \]
    has a dominated splitting.
    Hence, we may assume without loss of generality
    that $g$ is piecewise linear.
    Define a smooth monotonic function
    $\al:[0,1] \to [0,1]$ such that
    $\al(t_i) = t_i$ and
    \begin{math}
        \frac{d \al}{d t}|_{t=t_i} = 0  \end{math}
    for all $i$.
    Then
    \begin{math}
        (x,t) \mapsto g(x, \al(t))  \end{math}
    is a $C^1$ function.
\end{proof}
To keep consistent notation with \cref{sec:proofdacs},
the next proposition uses
$\Ec$ and $\Eu$ to denote the bundles of the dominated splitting.
In this context,
the $\Eu$ bundle may not necessarily be uniformly expanding
for all $t  \in  [0,1]$.

\begin{prop} \label{prop:slowcone}
    If $g:\bbT^2 \ti [0,1] \to \bbT^2$ is a $C^1$
    function such that each $g(\cdot, t)$
    has a dominated splitting,
    then there are cone fields for each $g(\cdot, t)$
    which vary continuously in $t$.

    In particular, there is $\eta > 0$ such that
    if, at a point $(x,t)  \in  \bbT^2 \ti [0,1]$,
    the dominated splitting is given by
    \[
        T_x \bbT^2 = \Ec(x,t) \oplus \Eu(x,t),
    \]
    then the cone $\Cone(x,t) \subof T_x \bbT^2$
    satisfies the properties
    \begin{enumerate}
        \item $\Eu(x,t) \subof \Cone(x,t)$,

        \item $\Ec(x,t) \subof \Cone^*(x,t)$, and

        \item $Dg(\Cone(x,s))$
        lies in the interior of
        $\Cone(g(x,t), t)$\\
        for all $s  \in  [0,1]$ with $|s - t| < \eta$.
    \end{enumerate}  \end{prop}
\begin{proof}
    If a diffeomorphism $f:M \to M$ has a dominated splitting
    of the form
    $TM = E_1 \oplus_< E_2$,
    then
    it is possible to define $\lam < 1$
    and a Riemannian metric $\vertiii{\cdot}$
    which depends smoothly on $x$
    and such that
    \[
        \frac{\vertiii{Df v_1}}{\vertiii{v_1}}
        < \lam
        \frac{\vertiii{Df v_2}}{\vertiii{v_2}}
    \]
    for all $x  \in  M$ and non-zero $v_1  \in  E_1(x)$ and $v_2  \in  E_2(x)$.
    With respect to this metric,
    the domination is seen under one 
    application of $f$ instead of an iterate $f^N$.
    See, for instance, \cite[Section 2.4]{cropot2015lecture}
    for a proof.
    By adapting the proof to our current setting,
    one can show that there is $\lam < 1$ and a
    smooth choice of metric $\vertiii{\cdot}$
    such that 
    \[
        \frac{\vertiii{Df v^c}}{\vertiii{v^c}}
        < \lam
        \frac{\vertiii{Df v^u}}{\vertiii{v^u}}
    \]
    for all $(x,t)  \in  \bbT^2 \ti [0,1]$
    and non-zero
    $v^c  \in  \Ec(x,t)$ and
    $v^u  \in  \Eu(x,t)$.
    Define
    \[
        \Cone(x,t) =
            \begin{bigset}
            v^c + v^u \: : \: v^c  \in  \Ec(x,t), \: 
                        v^u  \in  \Eu(x,t), \:
                        \vertiii{v^c}  \le  \vertiii{v^u}  \end{bigset}
    \]
    and note that
    $Dg(\Cone(x,t)$) is in the interior of
    $\Cone(g(x,t),t)$
    for all $x$ and $t$.
    As $\Ec$, $\Eu$, and $\vertiii{\cdot}$
    are all continuous,
    the function $(x,t) \mapsto \Cone(x,t)$
    is continuous.
    Therefore, at each $(x,t)$, there is $\eta > 0$ such that
    $Dg(\Cone(x,s)$) 
    lies in the interior of
    $\Cone(g(x,t), t)$
    for all $s  \in  [0,1]$ with $|s - t| < \eta$.
    Since the domain is compact,
    this $\eta$ may chosen uniformly.
\end{proof}
The next lemma is used to determine the effect of shearing
in the proof of \cref{thm:dacs}.

\begin{lemma} \label{lemma:zAz}
    Suppose $A:\bbT^2 \to \bbT^2$ is a hyperbolic toral automorphism
    which preserves the orientation of its stable bundle.
    Lift $A$ to a linear map on $\bbR^2$
    and let $\Eu_A(0)$ denote the lifted unstable manifold through the
    origin.
    For any $C > 1$, there is $z  \in  \bbZ^2$ such that
    \begin{math}
        \dist
        \bigl(
        \zeta \cdot A(z) + \xi \cdot z, \, \Eu_A(0)
        \bigr)
         \ge  C(\zeta+\xi)
    \end{math}
    for all $\zeta,\xi  \ge  0$.
\end{lemma}
The proof is left to the reader.
Note how the condition on orientation is necessary.

\medskip{}

To conclude the section,
we give a simple, concrete example of how
a linear Anosov map on $\bbT^2$ may be
isotoped into a derived-from-Anosov system.
This example has the nice additional property that the cone field
is independent of both $x$ and $t$.
For this example, assume $\bbT^2 = \bbR^2 / 2 \pi \bbZ^2$.

\begin{prop} \label{prop:bifurcate}
    For $t  \in  [0,1]$, let $g_t:\bbT^2 \to \bbT^2$ be defined by
    \[
            \begin{pmatrix}
            x \\
            y  \end{pmatrix}
        \mapsto
            \begin{pmatrix}
            5 & 2 \\
            2 & 1  \end{pmatrix}
            \begin{pmatrix}
            x - \tfrac{9 (1-t)}{10} \sin(x) & 0 \\
            0 & y - \tfrac{1-t}{2} \sin(y)  \end{pmatrix}
        .
    \]
    Then $g_t$ is weakly partially hyperbolic with a splitting of the form
    $\Ec \oplus \Es$.
    Moreover, if $\Cone$ is the cone field defined in each
    $T_x \bbT^2 \cong \bbR^2$
    by
    \begin{math}
            \begin{bigset}
            (u,v)  \in  \bbR^2 : u v  \ge  0  \end{bigset}
        ,
    \end{math}
    then
    $Dg_t(\Cone)$ lies in the interior of $\Cone$.

    If $t = 1$, then $g_t = g_1$ is a hyperbolic toral automorphism.

    If $t = 0$, then $g_t = g_0$ has a sink at $(x,y) = (0,0)$
    and is not Anosov.
              \end{prop}
\begin{proof}
    Most of this is routine multivariable calculus.
    The only minor difficulty is proving that $\Es$ is uniformly contracting.
    Suppose $(u_0,v_0)  \in  \bbR^2$ is a non-zero vector in the $\Es$ subbundle
    and write $(u_n,v_n) = Dg_t^n(u_0,v_0)$.
    Then
    \[
            \begin{pmatrix}
            u_1 \\
            v_1   \end{pmatrix}
        =
            \begin{pmatrix}
            5 & 2 \\
            2 & 1  \end{pmatrix}
            \begin{pmatrix}
            c_x & 0 \\
            0 & c_y   \end{pmatrix}
            \begin{pmatrix}
            u_0 \\
            v_0   \end{pmatrix}
    \]
    for some $c_x$, $c_y  \in  \bbR$ with $|c_x - 1| < \tfrac{9}{10}$
    and $|c_y - 1| < \tfrac{1}{2}$.
    Inverting gives
    \[
            \begin{pmatrix}
            u_0 \\
            v_0   \end{pmatrix}
        =
            \begin{pmatrix}
            c_x \inv & 0 \\
            0 & c_y \inv  \end{pmatrix}
            \begin{pmatrix}
            1 & -2 \\
            -2 & 5  \end{pmatrix}
            \begin{pmatrix}
            u_1 \\
            v_1   \end{pmatrix}
        .
    \]
    Since $\Es \cap \Cone = 0$, the product of $u_1$ and $v_1$ is negative.
    Assume without loss of generality that $u_1 < 0 < v_1$.
    Then
    \begin{math}
        v_0 = c_y \inv (5 v_1 - 2 u_1) > 3 v_1.
    \end{math}
    This shows that $v_n$ shrinks exponentially fast to zero
    as $n \to \infty$.
    For any point $p  \in  \bbT^2$ and non-zero vector $(u,v)$ in $\Es(p)$,
    the ratio $\tfrac{u}{v}$ is well defined and depends continuously on $x$.
    Therefore the ratio is uniformly bounded
    and so $u_n$ also converges exponentially quickly to zero.
\end{proof}
\section{Proof of \cref{thm:dacs}} \label{sec:proofdacs} %{{{1

We now construct the diffeomorphism in the conclusion 
of \cref{thm:dacs} and show that it is strongly partially hyperbolic.
Let $A:\bbT^2 \to \bbT^2$, $g_0:\bbT^2 \to \bbT^2$, and $\ep > 0$
be as in the theorem.
Choose constants $\ep/2 < a < b < c < d < e < \ep$.

We briefly give an intuitive description of the construction
before diving into the details.
The diffeomorphism $f$ will contract the region
$\bbT^2 \ti (-e,e)$ down towards $\bbT^2 \ti 0$.
In the region $\bbT^2 \ti [c,d]$,
a strong shear pushes
the vertical center direction
to be almost horizontal.
Then in the region $\bbT^2 \ti [a,b]$,
the dynamics in the horizontal direction
is changed from $A$ to $g_0$.
Finally in $\bbT^2 \ti [0,a)$,
the vertical contraction is dialled up
so that $\bbT^2 \ti 0$
is a normally attracting submanifold.
The effect of the dynamics on the $\Ecu$ and $\Eu$ subbundles
is shown in figure \ref{fig:planes}.
\begin{figure}
{
\vspace{1in}
\centering
\subfigure[The $\Ecu$ subbundle.]{\includegraphics{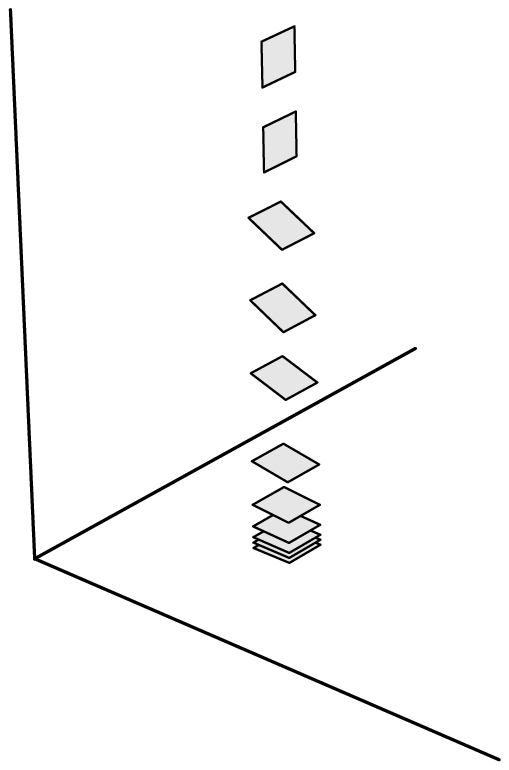}}
%\hfill
\subfigure[The $\Eu$ subbundle.]{\includegraphics{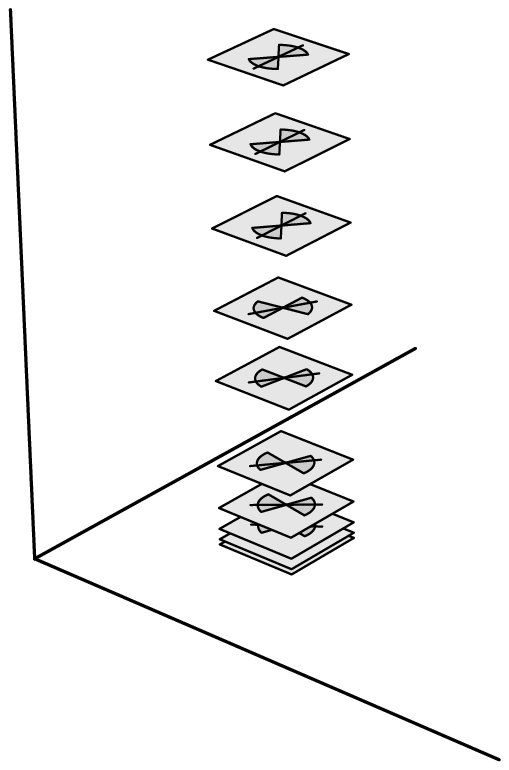}}
}
\vspace{0.5in}
\caption[Bundles under iteration]{
A depiction of the $\Ecu$ and $\Eu$ subbundles in the construction
given in this section.  
Consider a point $(x,s)  \in  \bbT^2 \ti (e,\ep)$
and its forward orbit $(x_n,s_n) := f^n(x,s)$.
For simplicity, we assume the sequence $\{x_n\}$ is constant.
The construction of $f$ ensures that $\{s_n\}$ decreases towards 0.
Subfigure (a) shows, for each $n  \ge  0$, the two-dimensional subbundle
$\Ecu(x_n,s_n)$.
When $s_n > d$, the $\Ecu$ subbundle is vertical.
When $c < s_n < d$, a shearing effect in the dynamics
pushes the $\Ecu$ planes to be
closer to horizontal. When $n$ is large and therefore $0 < s_n < a$,
the strong vertical contraction means the slopes of these planes tend to zero as
$n$ tends to $+\infty$.
Subfigure (b) shows, for each $n  \ge  0$,
the one-dimensional subbundle $\Eu(x_n,s_n)$
lying inside the horizontal
plane $T_{x_n} \bbT^2 \ti 0$.
It also depicts 
the cone field $\Cone(x_n,s_n)$ determined by \cref{prop:slowcone}.
Both the horizontal planes and $\Eu$
are unaffected by the shearing.
In the region $\bbT^2 \ti [a,b]$, the $\Eu$ direction moves around as different horizontal
maps $g(\cdot, t)$ are applied.
However, the $\Eu$ direction always stays within the cone field $\Cone$.
}
\label{fig:planes}
\end{figure}

Let $g:\bbT^2 \ti [0,1] \to \bbT^2$
be a $C^1$ function as in \cref{prop:isodom}.
By a reparameterization of the [0,1] coordinate,
assume without loss of generality that
$g(\cdot,t) = g_0$ for all $t  \le  a$
and $g(\cdot,t) = A$ for all $t  \ge  b$.
With $g$ now determined,
let $\eta > 0$ be as in \cref{prop:slowcone}.

Fix a value $\lam  \in  (0,1)$ such that
\begin{math}
    \|D g_0 v\| > 2 \lam
\end{math}
for all unit vectors $v$ in the tangent bundle of $\bbT^2$.
Define a smooth diffeomorphism $h:[0,\ep] \to [0,\ep]$
with the following properties.
\begin{enumerate}
    \item $h(s) = \lam s$ for all $s  \in  [0,\ep/2]$,

    \item $h(s) < s$ for all $s  \in  (0,e)$,

    \item $|h(s) - s| < \eta$ for all $s  \in  [a,b]$,

    \item $h^2(d) < c < h(d)$, and

    \item $h(s) = s$ for all $s  \in  [e,\ep]$.
\end{enumerate}
Define a smooth bump function $\rho:[0,\ep] \to [0,1]$
such that
\begin{enumerate}
    \item $\rho(s) = 0$ for all $s  \in  [0,c]$,

    \item $\rho'(s) > 0$ for all $s  \in  (c,d)$, and

    \item $\rho(s) = 1$ for all $s  \in  [d,\ep]$.
\end{enumerate}
Let $z$ be a non-zero element of $\bbZ^2$.
The precise conditions for choosing $z$
will be given later in this section.

With these objects in place, define
$f$ for $(x,s)  \in  \bbT^2 \ti [0,\ep]$ by
\[
    f(x,s) = (g(x,s) + \rho(s) \cdot z, h(s)).
\]
Extend $f$ to all of $\bbT^2 \times [-\ep,\ep]$
by the requirement that
$(y,t) = f(x,s)$ if and only if $(y,-t) = f(x,-s)$.
Finally, set
$f(x,s) = (A(x),s)$
for all $(x,s) \notin \bbT^2 \times [-\ep,\ep]$.

\medskip{}

We now consider the effect of $Df$ on vectors of the tangent bundle.
The identity $\bbT^3 = \bbT^2 \ti \bbS$
means that, for a point $p = (x,s)$,
a tangent vector $u  \in  T_p \bbT^3$ may be decomposed
as $u = (v,w)$
with $v  \in  T_x \bbT^2$ and $w  \in  T_s \bbS$.
During the proof, we will routinely
write vectors this way and refer to $v$ and
$w$ as the horizontal and vertical components
of $u$.
The linear toral automorphism $A$
gives a linear splitting
$T_x \bbT^2 = \Eu_A(x) \oplus \Es_A(x)$
which further defines subspaces
$\Eu_A(x) \ti 0$ and $\Es_A(x) \ti 0$ of $T_p \bbT^3$.
Also, if $\Cone(p) = \Cone(x,s) \subof T_x \bbT^2$
is the cone given by \cref{prop:slowcone},
then $\Cone(p) \ti 0$ may be considered
as a subset of $T_p \bbT^3$.

\begin{lemma} \label{lemma:highu}
    If $p = (x,s)  \in  \bbT^2 \ti [c,\ep]$
    and $y  \in  \bbT^2$ is such that $Df(p) = (y, h(s))$, 
    then
    \begin{math}
        Df_p(\Eu_A(x) \ti 0) = \Eu_A(y) \ti 0.
    \end{math}  \end{lemma}
\begin{proof}
    In this region,
    $f$ is given by
    \begin{math}
        f(x,s) = (A(x) + \rho(s) z, h(s))
    \end{math}
    and both $A$ and the translation $x \mapsto x + \rho(s) z$
    leave the linear unstable foliation of $A$ invariant.
\end{proof}
\begin{lemma} \label{lemma:lowcone}
    If $p  \in  \bbT^2 \ti [0,c]$,
    then
    \begin{math}
        Df(\Cone(p) \ti 0) \subof \Cone(f(p)) \ti 0.
    \end{math}  \end{lemma}
\begin{proof}
    This follows directly from the use of $\eta > 0$
    in the definition of $f$.
\end{proof}
\begin{lemma} \label{lemma:usplit}
    $f$ has a dominated splitting of the form
    $\Eu \oplus_> \Ecs$
    with $\dim \Eu = 1$.
\end{lemma}
\begin{proof}
    We will apply \cref{thm:genchaindom}
    with $Y = \bbT^2 \ti [0,e]$ and $Z = \bbT^2 \ti \{0,e\}$.
    Note that $Z$ has a well-defined partially hyperbolic splitting.
    If $p = (x,e)  \in  \bbT^2 \ti e$,
    then $\Eu_f(p) = \Eu_A(x) \ti 0$.
    If $p = (x,0)  \in  \bbT^2 \ti 0$,
    then
    $\Eu_f(p) = \Eu_{g_0}(x) \ti 0$.

    Consider an orbit $\{f^n(p)\}_{n  \in  \bbZ}$ where $p  \in  \bbT^2 \times (0,e)$.
    Up to shifting along the orbit,
    one may assume
    $p = (x,s)$
    with
    $s  \in  [h(c),c]$.
    Define $V_p \subof T_p \bbT^3$ by $V_p = \Eu_A(x) \ti 0$
    and let $u$ be a non-zero vector in $V_p$.
    Write $p_n = (x_s$, $s_n) = f^n(p)$ for all $n  \in  \bbZ$.
    First, consider the backwards orbit of $u$.
    By \cref{lemma:highu},
    \begin{math}
        u^{-m}  \in  \Eu_A(x_{-m}) \ti 0
    \end{math}
    for all $m > 0$.
    For a subsequence $\{m_j\}$,
    if $p_{-m_j}$ converges to a point $p_- = (x_-,e)$,
    then  $u^{-m_j}$ converges to a vector in
    $\Eu_A(x_-) \ti 0 = \Eu_f(p_-)$.

    Now consider the forward orbit of $u$.
    By \cref{lemma:lowcone},
    $u^n  \in  \Cone(x_n$, $s_n) \ti 0$
    for all $n > 0$.
    If a subsequence $\{p_{n_j}\}$
    converges to a point 
    $p_+ = (x_+,0)$
    and $v^{n_j}$ converges to a vector
    $v_+  \in  T_x \bbT^2 \ti 0$,
    then $v_+  \in  \Cone(p_+) \ti 0$.
    In particular,
    $v_+$ does not lie in $\Es_A(x_+) \ti 0$.

    This shows that the conditions of \cref{thm:genchaindom}
    are satisfied and a dominated splitting exists on all of
    $\bbT^2 \ti [0,e]$.
    By symmetry,
    a dominated splitting exists on $\bbT^2 \ti [-e,0]$.
    Since $f$ is linear outside of $\bbT^2 \ti [-e,e]$,
    there is a global dominated splitting
    on all of $\bbT^3$.
\end{proof}
%Suppose a subsequence $\{p_{n_j}\}$
%converges to a point 
%$p_+ = (x_+,0)$.
%Since $v^n$ is transverse to $\Ec_{g_0}$ for large $n$,
%$v^{n_j}$
%converges to a vector in $\Eu_{g_0}(x_+)$
%which shows that
%$u^{n_j}$ converges to a vector in $\Eu_f(p_+)$.

%If $N$ is such that both $s_N < a$
%and $|0 - s_N| < \eta$,
%then
%$u^n  \in  \Cone(x_n$, $0) \ti 0$
%for all $n > N$.
%Let $v^n$
%denote the horizontal component of $u^n$.
%Then the definition of $f$ implies that
%
%    $v^{n+1} = \normalize{Dg_0(v^n)}$
%
%for all $n > N$.
%Suppose a subsequence $\{p_{n_j}\}$
%converges to a point 
%$p_+ = (x_+,0)$.
%Since $v^n$ is transverse to $\Ec_{g_0}$ for large $n$,
%$v^{n_j}$
%converges to a vector in $\Eu_{g_0}(x_+)$
%which shows that
%$u^{n_j}$ converges to a vector in $\Eu_f(p_+)$.
%
%This shows that the conditions of \cref{thm:chaindom}
%are satisfied for any orbit in the region $\bbT^2 \ti (0,e)$.
%By symmetry,
%orbits in $\bbT^2 \ti (-e,0)$ also satisfy these conditions,
%and therefore the theorem
%implies the existence of a dominated splitting.

For a non-zero vector $u  \in  T \bbT^3$ with
horizontal component $v  \in  T \bbT^2$ and
vertical component $w  \in  T \bbS$,
define the \emph{slope} of $u$ by
\[
    \slope(u) =
    \frac{\|w\|}
        {\|v\|}  \in  [0,\infty].
\]
Note that $f$ maps a horizontal torus $\bbT^2 \ti s$
to a horizontal torus $\bbT^2 \ti h(s)$
and therefore $\slope(u) = 0$ implies that $\slope Df(u) = 0$.

\begin{lemma} \label{lemma:halfslope}
    If $p  \in  \bbT^2 \ti [0,\tfrac{\ep}{2}]$
    and $u  \in  T_p \bbT^3$
    with $\slope(u) < \infty$,
    then
    \[
        \slope Df(u) < \tfrac{1}{2} \slope(u).
    \]  \end{lemma}
\begin{proof}
    This follows from the choice of $\lam$
    at the start of the section.
\end{proof}
\begin{lemma} \label{lemma:slopects}
    There is $k  \ge  1$ and $\delta > 0$ such that
    if $p  \in  \bbT^2 \ti [h^3(d),h^2(d)]$ and
    $u  \in  T_p \bbT^3$ with
    $\slope(u) < \delta$,
    then
    $f^k(p)  \in  \bbT^2 \ti [0,\tfrac{\ep}{2}]$
    and
    $\slope Df^k(u) < 1$.
\end{lemma}
\begin{proof}
    Since $h(s) < s$ for all $s  \in  (0,e)$,
    there is $k  \ge  1$ so that
    $s < h^2(d)$ implies $h^k(s) < \ep/2$.
    Let $K$ be the compact set of all unit vectors
    based at points in
    $\bbT^2 \ti [h^3(d),h^2(d)]$,
    and let $K_0 \subof K$ be those vectors with slope zero.
    Define
    \[
        \gam:K \to [0,\infty], \, v \mapsto \slope Df^k(v).
    \]
    Since $\gam(K_0) = \{0\}$
    and $\gam$ is uniformly continuous,
    one may find $\delta > 0$ as desired.
\end{proof}
Since $h^2(d) < c$,
the choice of $z  \in  \bbZ^2$ does not affect the
definition of $f$ in the region $\bbT^2 \ti [0$, $h^2(d)]$.
Hence, the values $k$ and $\delta$ may be determined before specifying $z$.
The next lemma, however, does rely on this choice
and the conditions on $z$ are given in the lemma's proof.

\begin{lemma} \label{lemma:chooz}
    For any $\delta > 0$, the $z  \in  \bbZ^2$ used in the definition of $f$
    may be chosen such that the following property holds{:}

    If $p = (x,s)  \in  \bbT^2 \ti [h(d),d]$
    and $u \in \Eu_A(x) \ti T_s \bbS \subof T_p \bbT^3$,
    then $\slope Df^2(u) < \delta$.
\end{lemma}
\begin{proof}
    Write $u = (v,w)$ as before.
    If $w = 0$, then $\slope Df^2(u) = 0$.
    Therefore, one need only consider the case where
    $w$ is non-zero.
    Up to rescaling the vector $u$,
    assume $w$ is a unit vector pointing in the ``up''
    direction of $\bbS$.
    That is,
    pointing in the direction of increasing $s$.
    By calculating the derivative of
    \[
        f^2(x,s) =
        \bigl(
        A^2(x) + \rho(s) \cdot A(z) + (\rho \circ h)(s) \cdot z,
        \,
        h^2(s)
        \bigr)
    \]
    one can show that
    \[
        Df^2(v,w) = 
        \bigl(
        A^2(v) + \rho'(s) \cdot A(z) + (\rho \circ h)'(s) \cdot z,
        \,
        Dh^2(w)
        \bigr).
    \]
    %Note that $A^2(v)  \in  \Eu_A(y)$ for some $y  \in  \bbT^2$.
    Define
    \[
        \al := \min
            \begin{bigset}
            \rho'(s) + (\rho \circ h)'(s)
            : s  \in  [h(d), d]  \end{bigset}
    \]
    and
    \[
        \beta := \max
            \begin{bigset}
            (h^2)'(s)
            : s  \in  [h(d), d]  \end{bigset}
      \]
    and note that $\al > 0$.
    For some $C > 1$, if $z$ is given by \cref{lemma:zAz},
    then
    \begin{align*}
        \| A^2(v) \,+\, \rho'(s) \cdot A(z) \,+\, (\rho \circ h)'(s) \cdot z \|
        & \ge  \\
        \dist
        \bigl(
        \rho'(s) \cdot A(z) \,+\, (\rho \circ h)'(s) \cdot z, \, \Eu_A(0)
        \bigr)
        &> C \al
    \end{align*}
    and therefore $\slope Df^2(u) < \beta/C \al$.
    Take $C$ large enough that $\beta/C \al < \delta$.
\end{proof}
For the remainder of the proof,
assume $z$ was chosen so that \cref{lemma:chooz}
holds with $\delta > 0$ given by \cref{lemma:slopects}.
The last three lemmas then combine to show the following.

\begin{cor} \label{cor:dieslopedie}
    If $p = (x,s)  \in  \bbT^2 \ti [h(d),d]$
    and $u \in \Eu_A(x) \ti 0 \subof T_p \bbT^3$,
    then
    \[    
        \lim_{n \to +\infty} \slope Df^n(u) = 0.
    \]  \end{cor}
\begin{lemma} \label{lemma:cusplit}
    $f$ has a dominated splitting of the form
    $\Ecu \oplus_> \Es$
    with $\dim \Ecu = 2$.
\end{lemma}
\begin{proof}
    This proof follows the same general outline
    as the proof of \cref{lemma:usplit}.
    Let $Y$ and $Z$ be as in that proof.
    If $p = (x,e)  \in  \bbT^2 \ti e$,
    then $\Ecu_f(p) = \Eu_A(x) \ti T_e \bbS$.
    If $p = (x,0)  \in  \bbT^2 \ti 0$,
    then
    $\Eu_f(p) = T_x \bbT^2 \ti 0$.

    Now,
    consider an orbit $\{f^n(p)\}_{n  \in  \bbZ}$ where $p  \in  \bbT^2 \times (0,e)$.
    Up to shifting along the orbit,
    one may assume
    $p = (x,s)$
    with
    $s  \in  [h(d),d]$.
    Define $V_p \subof T_p \bbT^3$ by $V_p = \Eu_A(x) \ti T_s \bbS$
    and let $u$ be a non-zero vector in $V_p$.
    Write $p_n = (x_n$, $s_n) = f^n(p)$ for all $n  \in  \bbZ$.
    First, consider the backwards orbit of $u$.
    Note that
    \[
        Df^{n}(V_p) = \Eu_A(x_{n}) \ti T_{s_{n}} \bbS
    \]
    for all $n < 0$.
    Hence,
    if $\{u^{-m_j}\}$ is a convergent subsequence,
    then
    $p_{-m_j}$ converges to a point $p_-  \in  \bbT^2 \ti e$,
    and $u^{-m_j}$ converges to a vector in $\Ecu_f(p_-)$.
    In the other direction,
    \cref{cor:dieslopedie} implies that
    $\slope(u^n)$ tends to 0 as $n \to \infty$.
    If $\{u^{n_j}\}$ is a convergent subsequence,
    then
    $p_{n_j}$ converges to a point $p_+  \in  \bbT^2 \ti 0$,
    and $u^n_j$ converges to a vector in $\Ecu_f(p_+)$.
    One may then use
    \cref{thm:genchaindom}
    to show that the dominated splitting
    extends to all of $\bbT^3$.
\end{proof}
Now that the global invariant dominated splittings
$\Eu \oplus \Ecs$ and $\Ecu \oplus \Es$
are known to exist,
\cref{cor:nwineq} implies that $f$ is strongly partially hyperbolic
on all of $\bbT^3$.

\section{Further constructions} \label{sec:further} %{{{1

Rodriguez Hertz, Rodriguez Hertz and Ures gave two different constructions of
a system on the 3-torus with an invariant center-unstable 2-torus
\cite{rhrhu2016coherent}.
In the first of these constructions, the system is not dynamically coherent
as there is no invariant foliation tangent to $\Ec$.
In the second of their constructions,
the center bundle $\Ec$ is integrable, but not uniquely integrable.
The construction we gave in \cref{sec:proofdacs} corresponds to the first of
these cases.

\begin{prop} \label{prop:nondyn}
    The construction of $f$ given in \cref{sec:proofdacs} is not dynamically coherent.
\end{prop}
\begin{proof}
    The diffeomorphism $f$ leaves the foliation of
    horizontal planes invariant.
    Therefore, if a vector $u$ in the tangent bundle $T \bbT^3$
    has a non-zero vertical component, then $Df(u)$ also has
    a non-zero vertical component.
    If $p  \in  \bbT^2 \ti [h(d),e]$, and $u$ is a unit vector in $\Ec_f(p)$,
    then $u$ has a non-zero vertical component.
    By iterating forward, one sees that the same property
    holds for any $p  \in  \bbT^2 \ti (0,e]$.
    Hence, one may choose an orientation for the line bundle $\Ec_f$
    on $\bbT^2 \ti (0,e]$ so that the center direction always points in the direction
    of decreasing $s$.  That is, the orientation always points towards
    $\bbT^2 \ti 0$.
    
    This choice extends continuously to $\bbT^2 \ti [0,e]$.
    Further, by the symmetry of the contruction,
    the center orientation may be extended to $\bbT^2 \ti [-e,e]$,
    and on both sides, the center orientation points towards $\bbT^2 \ti 0$.
    This means that any parameterized curve
    $\gam : [0,+\infty) \to \bbT^3$
    that starts in $\bbT^2 \ti 0$, stays tangent to $\Ec$,
    and agrees with the orientation of $\Ec$,
    must remain for all time inside of $\bbT^2 \ti 0$.

    The constructed $f$ is homotopic to $A$ times the identity map on $\bbS$.
    If $f$ were dynamically coherent,
    then by the leaf conjugacy given in \cite[Theorem 1.3]{hp2014pointwise},
    there would be a circle tangent
    to $\Ec_f$ though every point in $\bbT^3$.
    In particular, there would be an invariant foliation of
    center circles lying in $\bbT^2 \ti 0$.
    As the dynamics $g_0$ on $\bbT^2 \ti 0$
    is homotopic to a hyperbolic toral automorphism,
    this is not possible and gives a contradiction.
\end{proof}
We now look at ways in which the construction in the previous section may be modified.
The definition of $f$ may be stated piecewise as
\[
    f(x,s) =
        \begin{cases}
        \big( g(x,s) + \rho(s) \cdot z, \; h(s) \big),
            & \text{if } s  \in  [0,\ep] \\
        \big( g(x,-s) + \rho(-s) \cdot z, \; -h(-s) \big),
            & \text{if } s  \in  [-\ep,0] \\
        \big( Ax, \; s \big),
            & \text{if } s \notin [-\ep,\ep].
          \end{cases}  \]
Recall that $z  \in  \bbZ^2$ was chosen to satisfy the conclusions of \cref{lemma:zAz}.
If $k$ is any non-zero integer, then the product $k \cdot z  \in  \bbZ^2$ also satisfies
those same conclusions.
Thus, for any choice of non-zero integers $k_1$ and $k_2$,
one may show that
the function defined by
\[
    (x,s) \mapsto
        \begin{cases}
        \big( g(x,s) + k_1 \rho(s) \cdot z, \; h(s) \big),
            & \text{if } s  \in  [0,\ep] \\
        \big( g(x,-s) + k_2 \rho(-s) \cdot z, \; -h(-s) \big),
            & \text{if } s  \in  [-\ep,0] \\
        \big( Ax, \; s \big),
            & \text{if } s \notin [-\ep,\ep]  \end{cases}
\]
is partially hyperbolic with a $cu$-torus at $\bbT^2 \ti 0$.
        
The choices of sign for $k_1$ and $k_2$
give four different ways to realize $g_0$ as the dynamics on an invariant
$cu$-torus.
These correspond to the two different ways the center bundle
can approach a horizontal direction on either side of $\bbT^2 \ti 0$
and are depicted in \cref{fig:clines}.
\begin{figure}
{
\centering
\subfigure[]{\includegraphics{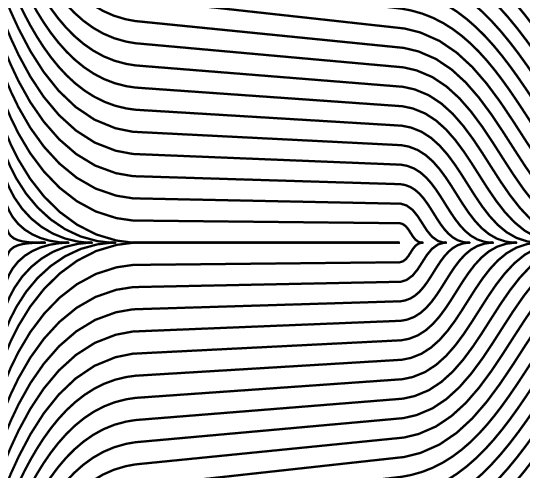}}
\qquad
\subfigure[]{\includegraphics{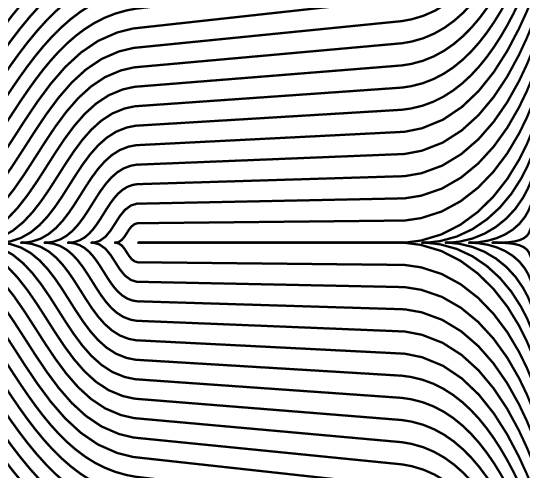}}
\\
\subfigure[]{\includegraphics{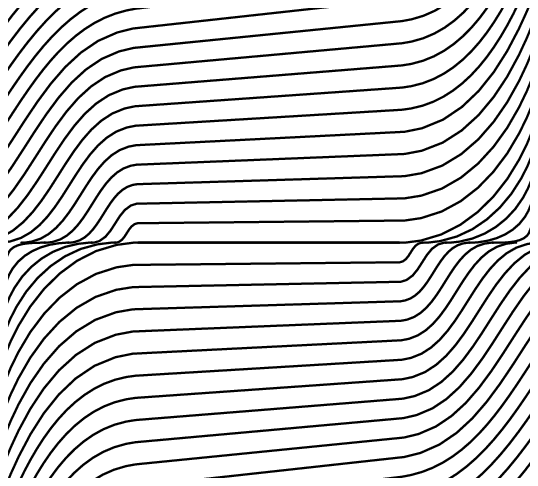}}
\qquad
\subfigure[]{\includegraphics{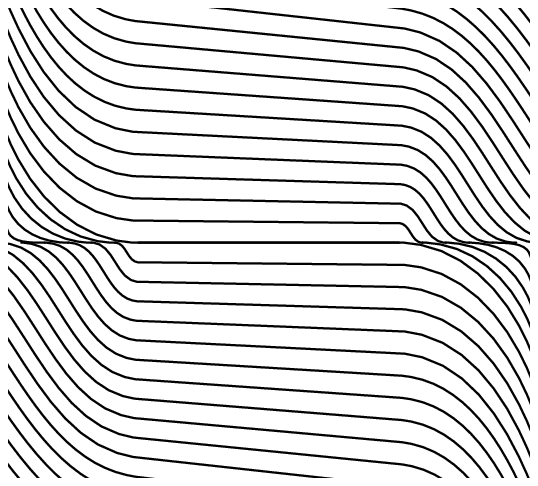}}
}
\caption[Direction]{
Four possible ways in which the center bundle may behave near a
center-unstable torus with derived-from-Anosov dynamics.
Shown here are lines tangent to the $\Ec$ direction inside a $cs$-leaf.
In each subfigure, the $cs$-leaf intersects the $cu$-torus
in a horizontal line passing through the middle of the subfigure.
In this example, the middle of this line intersects the basin of repulsion
of a repelling fixed point inside the $cu$-torus
so that there are no cusps here.
}
\label{fig:clines}
\end{figure}
The cases (a) and (b) in the figure are not dynamically
coherent, as may be shown by the argument in the proof of \cref{prop:nondyn}.
From the figure, it appears that the dynamics depicted in each of cases
(c) and (d) has an invariant center foliation
with leaves which
topologically cross the torus.
Rigorously proving the existence of this center foliation
will require a sophisticated analysis of the Franks
semiconjugacy of the system and its relation to the branching foliations of
Brin, Burago and Ivanov.  This work is left to a future paper.

\medskip{}

The above modifications to the construction 
suggest a way to prove \cref{thm:everycs} in the case where $g_0$
reverses the orientation of $\Ec$.

\begin{proof}[Proof of \cref{thm:everycs}]
    Let $g_0$ be weakly partially hyperbolic with a splitting
    of the form $\Ec \oplus \Eu$.
    The case where $g_0$ preserves the orientation of $\Ec$
    was already handled in \cref{sec:proofdacs},
    so assume here that $g_0$ reverses the center orientation.
    Then $g_0$ is homotopic to a hyperbolic toral automorphism
    $A$ which reverses the orientation of its stable bundle $\Es_A$.
    Analogously to \cref{lemma:zAz}, 
    for any $C > 1$, there is $z  \in  \bbZ^2$ such that
    \begin{math}
        \dist
        \bigl(
        \zeta \cdot A(z) - \xi \cdot z, \, \Eu_A(0)
        \bigr)
         \ge  C(\zeta+\xi)
    \end{math}
    for all $\zeta,\xi  \ge  0$. (Note now the minus sign before $\xi \cdot z$.)
    
    Our constructed diffeomorphism on $\bbT^3$ will be the result
    of modifying the linear map
    $A \ti (-\id)$ defined on $\bbT^2 \ti \bbS$.
    Fix a small $\ep > 0$ and
    define $h:[0,\ep] \to [0,\ep]$
    and $\rho:[0,\ep] \to [0,1]$
    with the properties as listed
    in \cref{sec:proofdacs}.
    Define $f$ by
    \[
        f(x,s) =
            \begin{cases}
            \big( g(x,s) + \rho(s) \cdot z, \; -h(s) \big),
                & \text{if } s  \in  [0,\ep] \\
            \big( g(x,-s) - \rho(-s) \cdot z, \; h(-s) \big),
                & \text{if } s  \in  [-\ep,0] \\
            \big( Ax, \; -s \big),
                & \text{if } s \notin [-\ep,\ep].  \end{cases}
    \]
    If $s  \in  [h(d),d]$, then
    \[
        f^2(x,s) =
        (A^2(x) + \rho(s) \cdot A(z) - (\rho \circ h)(s) \cdot z, h^2(s)).
    \]
    The above analogue of \cref{lemma:zAz} then establishes an analogue
    of \cref{lemma:chooz} in this context.
    The other parts of the proof in \cref{sec:proofdacs} are also easily
    adapted and one may show that $f$ is strongly partially hyperbolic.
\end{proof}
For simplicity, the previous section constructed a diffeomorphism on $\bbT^3$.
It is a simple matter to apply the same techniques to a 3-manifold
defined by the suspension
of either an Anosov map or ``minus the identity'' on $\bbT^2$.  
The important condition in each case is that there is a
strongly partially hyperbolic
map and an invariant subset of the manifold homeomorphic to $\bbT^2 \ti [-\ep,\ep]$
where the dynamics is given by $A \ti \id$.
As shown in \cite{RHRHU-tori}, these are the only orientable 3-manifolds
which allow a torus tangent to $\Ecu$ or $\Ecs$.

As explored in \cite[Section 4]{BonattiWilkinson} and
\cite[Appendix A]{hp2015classification},
it is possible to define partially hyperbolic diffeomorphisms on
non-orientable manifolds which are double covered by the 3-torus.
A similar construction works in the current setting
to define one-sided center-stable and center-unstable tori.

\begin{prop} \label{prop:onesided}
    For any weakly partially hyperbolic diffeomorphism
    $g_0:\bbT^2 \to \bbT^2$
    which preserves its center orientation,
    there is a non-orientable 3-manifold $M$, 
    an embedding $i:\bbT^2 \to M$
    and a strongly partially hyperbolic diffeomorphism
    $f:\bbT^3 \to \bbT^3$
    such that the one-sided torus $i(\bbT^2)$
    is tangent either to $\Ecs_f$ or $\Ecu_f$
    and $i \inv \circ f \circ i = g_0$.
\end{prop}
\begin{proof}
    Assume $g_0$ has a splitting of the form $\Eu \oplus \Ec$
    and construct $f:\bbT^3 \to \bbT^3$ as in \cref{sec:proofdacs}.
    Assume $\bbT^3$ is defined as $\bbR^3/\bbZ^3$
    and lift $f$ to a map $\tf:\bbR^3 \to \bbR^3$
    such that $\tf(\bbR^2 \ti 0) = \bbR^2 \ti 0$.
    Construct a new closed 3-manifold by quotienting $\bbR^3 = \bbR^2 \ti \bbR$
    by the group generated by the translations
    $(v,s) \mapsto (v, s+1)$
    and
    $(v,s) \mapsto (v+(0,1), s)$
    and the isometry
    $(v,s) \mapsto (v+(1,0), -s)$.
\end{proof}

This concludes our construction of examples in dimension 3.
The rest of paper handles constructions in higher dimension.

\section{Compact center-stable manifolds of higher dimension} \label{sec:higher} %{{{1

This section proves \cref{thm:allsinkssimple}.
In fact, we will prove the following restatement of the theorem which gives
more technical details about the nature of the constructed diffeomorphism $F$.

\begin{prop} \label{prop:allsinks}
    Let $f_0 : M \to M$ be a diffeomorphism,
    let $X \subof M$ be a finite invariant set such that
    every $x  \in  X$ is either a periodic source or sink,
    and let $U$ be a neighborhood of $X$.
    Then,
    there are a diffeomorphism $f : M \to M$,
    a toral automorphism $A : \bbTD \to \bbTD$,
    a smooth map $h : M \to \bbTD$,
    and a diffeomorphism $F : M \ti \bbTD \to M \ti \bbTD$
    defined by
    \[
        F(x, v) = (f(x), Av + h(x))
    \]
    such that{:}
    \begin{enumerate}
        \item $F$ is strongly partially hyperbolic;

        \item A is a linear Anosov diffeomorphism with
        $\dim \Es_A = \dim \Eu_A = \dim M$;

        \item $f(x) = f_0(x)$ and $h(x) = 0$ for all $x  \in  M \sans U$;

        \item if $x  \in  NW(f) \sans X$ and $v  \in  \bbTD$, then
        \[            \Es_F(x,v) = 0 \oplus \Es_A(v),
            \quad
            \Ec_F(x,v) = T_x M \oplus 0,
            \qandq
            \Eu_F(x,v) = 0 \oplus \Eu_A(v);
        \]
        \item if $x  \in  X$ is a sink and $v  \in  \bbTD$, then
        \[            \Es_F(x,v) = T_x M \oplus 0,
            \quad
            \Ec_F(x,v) = 0 \oplus \Es_A(v),
            \qandq
            \Eu_F(x,v) = 0 \oplus \Eu_A(v);
        \]
        \item if $x  \in  X$ is a source and $v  \in  \bbTD$, then
        \[            \Es_F(x,v) = 0 \oplus \Es_A(v),
            \quad
            \Ec_F(x,v) = 0 \oplus \Eu_A(v),
            \qandq
            \Eu_F(x,v) = T_x M \oplus 0.
        \]  \end{enumerate}  \end{prop}
Note that the notation and, in particular,
the functions $f$, $g$, and $h$ play very different roles here than
in previous sections.

The basic idea of the construction is to replace the possibly non-linear
behaviour of $f_0$ in a neighbourhood of a point $x  \in  X$ with a simple linear
contraction or expansion.  Then, both $f$ and $A$ are linear maps and there are
exactly three rates of contraction or expansion given by $f$ and the stable and
unstable directions of $A$.  This allows us to restrict our consideration to
the case of a linear map
\[
        F(w, x, y) = (\lam \inv w, b x, \lam y)
\]
defined on $\bbRd \ti \bbRd \ti \bbRd$ and where $0 < \lam < b < 1$.
We deform this map so that a $d$-dimensional subspace which lies roughly
in the direction of $0 \ti \bbRd \ti 0$
converges to the subspace $0 \ti 0 \ti \bbRd$
under application of the derivative $DF^n$ as $n \to +\infty$.
This provides the effect of pushing the center direction into the stable
direction of $A$.

The first step is to establish the following.

\begin{lemma} \label{lemma:shearup}
    For $0 < \lam < b < 1$ and $C > 1$,
    there is a diffeomorphism $f$ of\, $\bbRd$ and a smooth map
    $h : \bbRd \to \bbRd$
    such that the diffeomorphism $F$ of\, $\bbRd \ti \bbRd$
    defined by
    \[
        F(x,y) = (f(x), \lam y + h(x) )
    \]
    has the following properties.
    If $p = (x,y)  \in  \bbRd \ti \bbRd$ with $b  \le  \|x\|  \le  1$,
    then
    \begin{enumerate}
        \item $f(x) = b x$ and $h(x) = 0$; and

        \item if $V \subof \bbRd \ti \bbRd$ is the graph of a linear map
        $L : \bbRd \to \bbRd$
        with $\|L\| < C$,
        then
        $DF^n_p(V)$ tends to $0 \ti \bbRd$ as $n$ tends to $+\infty$.
    \end{enumerate}  \end{lemma}
As an aid in proving \cref{lemma:shearup}, we first introduce a notion of
the ``quality'' of a square matrix.
This is closely related to the idea of a row diagonally dominated matrix,
however we use different wording here in order to avoid potential confusion
between different notions of domination.

Let $A$ be a $d \ti d$ matrix with entries $a_{i j}$.
Define the \emph{quality} of the matrix as
\[
    q(A) :=
        \frac { \min \{ a_{i i} \: : \: 1  \le  i  \le  d \} }
             { \sum \{ |a_{i j}| \: : \: 1  \le  i,j  \le  d, \: i  \ne  j \} }.
\]
To have positive quality, a matrix must have positive diagonal entries.
We allow $q(A) = +\infty$ 
which occurs if and only if $A$ is diagonal and positive
definite.

\begin{lemma} \label{lemma:gershgorin}
    If $q(A) > 2$, then
    $A$ is invertible and the operator norm of the inverse satisfies
    \[
        \left\|A \inv \right\|  \le 
        \max \left\{ \frac{2 d}{a_{i i}} : 1  \le  i  \le  d \right\}.
    \]  \end{lemma}
\begin{proof}
    This is a variation on the Gershgorin circle theorem.
    Suppose $v  \in  \bbRd$ is non-zero and
    let $i$ be an index such that $|v_i|  \ge  |v_j|$ for all $j$.
    Then,
    \begin{align*}
            \left| \sum_{j=1}^d a_{i j} v_j \right|
            \,  \ge  \,
            \left( a_{i i} - \sum_{j  \ne  i} |a_{i j}| \right) |v_i|
            \,  \ge  \,
            \tfrac{1}{2} a_{i i} |v_i|
    \end{align*}
    which implies that
    \begin{math}
        \|A v\|
             \ge 
            \tfrac{1}{2d} a_{i i} \|v\|.
    \end{math}  \end{proof}
\begin{lemma} \label{lemma:bolicity}
    If $A$ is a $d \ti d$ matrix with $q(A) > 0$
    and $B$ is a positive definite diagonal matrix with entries $b_{i j}$,
    then
    \[
        q(A B)  \ge 
            q(A) \,
            \min
            \left\{
                \frac{b_{i i}}{b_{j j}} : 1  \le  i,j  \le  d
            \right\}.
    \]  \end{lemma}
\begin{proof}
    Multiply $A$ and $B$ and check.
\end{proof}
\begin{proof}[Proof of \cref{lemma:shearup}]
    We prove \cref{lemma:shearup} in the specific case where
    \[
        \frac{b - \lam}{b - 1} < \lam.
    \]
    Showing that the general case of $\lam < b < 1$ may be proved from this
    special case is left to the reader.
    With this assumption added, there is a constant $0 < a < \lam$
    such that
    \[
        \frac{b - a}{b - 1} < a.
    \]
    Define a function
    \begin{math}
        g_0 : [0,\infty) \to [a,b]  \end{math}
    such that
    \begin{enumerate}
        \item $g_0(t) = a$ for $t  \le  b$,

        \item $g_0(t) = b$ for $t  \ge  1$,
        and

        \item \begin{math}
            0  \le  t g_0'(t) < a  \end{math}
        for all $t  \ge  0$.
    \end{enumerate}
    Define a smooth bump function $\rho:[0,\infty) \to [0,1]$
    with $\rho(t) = 0$ for $t  \ge  b$, and $\rho(t) = 1$ for $t  \le  b^2$.
    Define $h : \bbRd \to \bbRd$
    by $h(x) = \rho(\|x\|) x$.

    Before defining $f$,
    we first consider the behaviour of
    \begin{math}
        \hF(x,y) := (b x, \lam y + h(x))  \end{math}
    under iteration.
    Let $p = (x,y)$, $V$, and $L$ be as in item (2) of
    the statement of the lemma being proved.
    In particular, $b  \le  \|x\|  \le  1$.
    For $n  \ge  0$, define $\hV_n := D \hF^n_p(V)$
    and let $\hL_n : \bbRd \to \bbRd$ be the linear map
    such that $\graph(\hL_n) = \hV_n$.
    The definition of $\hF$ implies that
    \[
        \hL_{n+1} = \tfrac{\lam}{b} \hL_n + \tfrac{1}{b} Dh
    \]
    where the derivative $Dh$ is evaluated at $b^n x$.
    If $n > 2$, then $Dh$ is the identity map, $I$, and
    \[
        \hL_{n+1} = \tfrac{\lam}{b} \hL_n + \tfrac{1}{b} I.
    \]
    It follows that $\hL_n$ converges exponentially fast to $(b - \lam) \inv I$.
    When viewed as a matrix, $(b - \lam) \inv I$
    is diagonal and positive definite and so its ``quality,''
    as defined above, is $q((b - \lam) \inv I) = +\infty$.
    Therefore, there is $N > 2$ such that $q(\hL_n) > 4$ for all $n  \ge  N$.
    By compactness,
    one may find a uniform value of $N$ such that
    this lower bound on $q(\hL_n)$ holds for any starting $p = (x,y)$, $V$, and $L$
    with $\|L\| < C$.
    
    With $N$ now fixed, define $g : [0, \infty) \to [a,b]$
    by $g(t) := g_0(b^{-N} t)$
    and observe that
    \begin{enumerate}
        \item $g(t) = a$ for $t  \le  b^{N+1}$,

        \item $g(t) = b$ for $t  \ge  b^N$,
        and

        \item \begin{math}
            0  \le  t g'(t) < a  \end{math}
        for all $t  \ge  0$.
      \end{enumerate}
    Define $f$ by $f(x) = g(\|x\|) x$.
    With $f$ and $h$ now defined,
    we show that $F(x,y) = (f(x)$, $\lam y + h(x))$
    satisfies the conclusions of the lemma.

    This definition of $F$ has a form of radial symmetry{:}
    if $R$ is a rigid rotation about the origin in $\bbRd$,
    then
    $f \circ R = R \circ f$,
    $h \circ R = R \circ h$,
    and $F \circ (R \ti R) = (R \ti R) \circ F$.
    Further, any one-dimensional subspace in $\bbRd$ is invariant under $f$.
    Because of this symmetry,
    when analysing orbits of $F$,
    we need only consider points of the form
    $p = (x,y)$ where $x  \in  \bbR \ti 0$.
    That is, if $x$ is written in coordinates as
    $x = (x_1$, $x_2, \ldots, x_d)$,
    then $x_2 = x_3 = \cdots = x_d = 0$.

    The partial derivatives of $f : \bbRd \to \bbRd$
    are given by
    \[
        \frac{\del f_i}{\del x_j}
        =
        g(\|x\|) \delta_{i j}
        +
        \frac{x_i x_j}{\|x\|} g'(\|x\|).
    \]
    Since we are assuming $x  \in  \bbR \ti 0$,
    the terms $x_i x_j$ all evaluate to 0 except for the term
    $x_1 x_1$.
    Therefore
    \begin{align*}
        \frac{\del f_1}{\del x_1}
        &=
        g(\|x\|)
        +
        \|x\| g'(\|x\|)
        &&\\
        \frac{\del f_i}{\del x_i}
        &=
        g(\|x\|)
        &&\text{ if $i > 1$, and}
        \\
        \frac{\del f_i}{\del x_j}
        &=
        0
        && \text{ if } i  \ne  j.
    \end{align*}
    Further $g'(\|x\|)$ is non-zero only when
    $b^{N+1} < \|x\| < b^N$ and
    one may show that
    \[        g(\|x\|)  \le  g(\|x\|) + \|x\| g'(\|x\|)  \le  2 g(\|x\|).
    \]
    In other words, the Jacobian of $f$ is a diagonal matrix where no entry is
    more than twice as large as any other.

    Let $p = (x,y)$ with $x  \in  \bbR \ti 0$ and $b  \le  \|x\|  \le  1$.
    Let $V$ and $L$ be as in item (2) of the statement of the lemma.
    For $n  \ge  0$, define $V_n := DF^n_p(V)$
    and $L_n : \bbRd \to \bbRd$ such that $\graph(L_n) = V_n$.
    We now analyze $L_n$ as $n$ tends to $+\infty$.
    First, if $n < N$,
    then $\|f^n(x)\|  \ge  b^N$
    and the functions $F^n$ and $\hF^n$ are equal in a neighborhood of $p$.
    Therefore $L_N = \hL_N$
    and in particular $q(L_N) > 4$.

    For the case $n = N$,
    the equality $\graph(L_{N+1}) = DF(\graph(L_N))$
    may be written as
    \[    
            \begin{bigset}
            (u, L_{N+1}(u)) : u  \in  \bbRd  \end{bigset}
        =
            \begin{bigset}
            (Df(v), \lam L_N(v) + v) : v  \in  \bbRd  \end{bigset}
    \]
    showing that $L_{N+1} = (\lam L_N + I) \circ Df \inv$
    where $Df$ is evaluated at $f^N(x)$.
    \Cref{lemma:bolicity}, along with the above remark about the Jacobian of $f$,
    shows that 
    \[
        q \bigl( (\lam L_N + I) \circ Df \inv \bigr)
        \,  \ge  \,
        \tfrac{1}{2} q \bigl( \lam L_N + I \bigr).
    \]
    and this implies that
    $q(L_{N+1})  \ge  \tfrac{1}{2} q(L_N) > 2$.

    Finally, for $n > N$,
    the point $f^n(x)$ satisfies $\|f^n(x)\|  \le  b^{N+1}$.
    For points in this region, $Df = a I$ and so
    \begin{math}
        L_{N+1} = \tfrac{\lam}{a} L_n + \tfrac{1}{a} I.
    \end{math}
    which implies that
    $q(L_{n+1}) > q(L_n) > 2$
    for all large $n$.
    Since $\tfrac{\lam}{a} > 1$, the linear map $L_n$ when viewed as a matrix
    has positive entries on its diagonal and these entries tend to $+\infty$
    as $n$ tends to $+\infty$.
    \Cref{lemma:gershgorin} implies that $\|L_n \inv\|$ tends to zero as $n \to +\infty$
    and therefore the sequence of subspaces $V_n$ tends to $0 \ti \bbRd$.  \end{proof}
%\begin{comment}

The next result simply adds an expanding direction to \cref{lemma:shearup}.

\begin{cor} \label{cor:fullshearup}
    For $0 < \lam < b < 1$ and $C > 1$,
    there is a diffeomorphism $f$ of\, $\bbRd$ and a smooth map
    $h : \bbRd \to \bbRd$
    such that the diffeomorphism $F$ of\, $\bbRd \ti \bbRd \ti \bbRd$
    defined by
    \[
        F(w, x, y) = (\lam \inv w, f(x), \lam y + h(x) )
    \]
    has the following properties.
    If $p = (w,x,y)  \in  \bbRd \ti \bbRd \ti \bbRd$ with $b  \le  \|x\|  \le  1$,
    then
    \begin{enumerate}
        \item $f(x) = b x$ and $h(x) = 0$;

        \item if $V \subof \bbRd \ti \bbRd \ti \bbRd$ is the graph of a linear map
        $L : \bbRd \to \bbRd \ti \bbRd$
        with $\|L\| < C$,
        then
        $DF^n_p(V)$ tends to $\bbRd \ti 0 \ti 0$ as $n$ tends to $+\infty$;
        and
    
        \item if $V \subof \bbRd \ti \bbRd \ti \bbRd$ is the graph of a linear map
        $L : \bbRd \ti \bbRd \to \bbRd$
        with $\|L\| < C$,
        then
        $DF^n_p(V)$ tends to $\bbRd \ti 0 \ti \bbRd$ as $n$ tends to $+\infty$.
      \end{enumerate}  \end{cor}
\begin{proof}
    Use the same $f$ and $h$ as in \cref{lemma:shearup}.
\end{proof}
With this established, we now consider diffeomorphisms defined on closed
manifolds.
For a closed manifold $M$ and a hyperbolic toral automorphism
$A : \bbTD \to \bbTD$,
an \emph{A-map} is a map
\begin{math}
    F : M \ti \bbTD \to M \ti \bbTD  \end{math}
of the form 
\[
    F(x,v) = (f(x), A v + h(x)).
\]
See \cite{goh2015partially} for a more general definition and further
details.
If $F$ is also a (strongly) partially hyperbolic diffeomorphism,
we call it a partially hyperbolic $A$-map.
Note that we do not a priori assume that the partially hyperbolic
splitting has any relation to the fibers of the torus bundle.

There is a small subtlety in proving \cref{prop:allsinks}
in the case where the basin of a sink overlaps the basin of a source.
To handle this, we will prove \cref{prop:allsinks}
by induction and keep track of a property we call being
``graph like'' for the splitting at a point.

For a partially hyperbolic $A$-map
and a point $x  \in  M$,
the subbundle $\Eu$ is \emph{graph like}
at $x$ if, for all $v  \in  \bbTD$,
$\Eu(x,v)$ is the graph of a linear function from $\Eu_A(v)$ to
$\Es_A(v) \oplus T_x M$.
Similarly,
$\Ecu$, $\Ecs$, and $\Es$ are graph like at $x$ if they are graphs of linear functions
\[
    T_x M \oplus \Eu_A(v) \to \Es_A(v),
    \quad
    T_x M \oplus \Es_A(v) \to \Eu_A(v),
    \qandq
    \Es_A(v) \to \Eu_A(v) \oplus T_x M
\]
respectively.
If all of $\Eu$, $\Ecs$, $\Ecu$, and $\Es$ are graph like at $x$,
we say the splitting is graph like at $x$.

Since the bundles in the splitting are continuous and $DF$-invariant 
the following is easily verified.

\begin{lemma} \label{lemma:graphopen}
    Let $F$ be a partially hyperbolic A-map with base map
    $f : M \to M$.
    For a bundle $E  \in  \{\Eu, \Ecu, \Ecs, \Es\}$,
    the set of graph-like points is open and $f$-invariant.
\end{lemma}
Next, we consider a normally attracting fiber.

\begin{lemma} \label{lemma:verthyp}
    For a partially hyperbolic A-map $F$ with base map
    $f : M \to M$,
    if $x  \in  M$ is a periodic sink for $f$ and
    $x \ti \bbTD$ is tangent to $\Ecu$,
    then $\Es$ and $\Ecs$ are graph like
    for every point in the basin of $x$.
\end{lemma}
\begin{proof}
    Since $\Es_F$ is transverse to $x \ti \bbTD$,
    it is graph like at $x$.
    By the uniqueness of the dominated splitting on $x \ti \bbTD$,
    \begin{math}
        \Ec_F(x, v) = \Es_A(v)
    \end{math}
    for all $v  \in  \bbTD$.
    Therefore, $\Ecs_F$ is also graph like at $x$.
    By the previous lemma,
    being graph like at $x$ extends to being graph like
    on the basin of $x$.
\end{proof}
The next lemma allows us to replace non-linear sinks with linear ones.

\begin{lemma} \label{lemma:sinkpaste}
    Let $f_0 : M \to M$ be a diffeomorphism
    with a periodic sink $x_0 = f_0^k(x_0)$
    and let $\ep > 0$ and $0 < b < 1$.
    Then there is a diffeomorphism $f : M \to M$
    and a coordinate chart
    $\varphi : [-1, 1]^d \to M$
    such that
    \begin{enumerate}
        \item if dist(x, $x_0) > \ep$, then
        $f(x) = f_0(x)$,

        \item $f$ and $f_0$ have the same non-wandering set,

        \item $\varphi(0) = x_0$, and

        \item $\varphi \inv \circ f^k \circ \varphi(y) = b y$
        for all $y  \in  [-1,1]^d$.
    \end{enumerate}  \end{lemma}
\begin{proof}
    This follows from standard methods of
    pasting diffeomorphisms \cite{wil1972pasting}.
    First, one may make a $C^1$ small perturbation in order to
    assume that $\varphi \inv \circ f^k \circ \varphi$ is linear
    in a neighborhood of 0.
    Then, deform the linear map inside that neighborhood
    to get the desired homothety.
\end{proof}
Now we state what will be the inductive step in proving \cref{prop:allsinks}.

\begin{prop} \label{prop:addcu}
    Let A be a hyperbolic toral automorphism of $\bbT^D$
    with eigenvalues $\lam < 1$ and $\lam \inv > 1$,
    each of multiplicity $d = \tfrac{1}{2} D$.
    Suppose $F_0$ is a partially hyperbolic A-map
    having a base map $f_0 : M \to M$
    with $\dim M = d$
    and $x_0$ is a periodic sink
    such that the splitting is graph like at $x_0$.
    For any $\ep > 0$, there is a partially hyperbolic A-map
    $F$ such that
    \begin{enumerate}
        \item if dist(x, $x_0) > \ep$,
        then
        $F(x,v) = F_0(x,v)$
        for all $v  \in  \bbTD$;

        \item if the splitting for $F_0$ is graph like at
        $x  \ne  x_0$, then the splitting for $F$ is also
        graph like at $x$;
        and

        \item $x_0 \ti \bbTD$ is an $F$-periodic submanifold tangent to $\Ecu_F$.
    \end{enumerate}  \end{prop}
\begin{proof}
    This proof breaks into two steps.
    First, we deform $F_0$ to produce a partially hyperbolic map $F_1$
    which is linear in a neighborhood of $x_0 \ti \bbTD$,
    but which still has a graph like splitting at $x_0$.
    Then, we paste in the dynamics given by \cref{cor:fullshearup},
    to produce a partially hyperbolic map $F$ for which
    $\Ecu_F$ is tangent to $x_0 \ti \bbTD$.

    Let $U$ be a neighborhood of the orbit of $x_0$
    such that $\barU$ is contained in the basin of attraction
    and $f_0(\barU) \subof U$.
    Define a smooth function $h_1 : M \to \bbTD$
    such that $h_1(x) = h_0(x)$
    for all $x  \in  M \sans U$
    and $h_1(x) = 0$ for all $x  \in  f(U)$.

    Fix $b$ such that $\lam < b < 1$ where
    $\lam$ is the stable eigenvalue of $A$.
    Let $k$ denote the period of $x_0$.
    By \cref{lemma:sinkpaste},
    there is a coordinate chart
    $\varphi : [-1,1]^d \to M$
    and a diffeomorphism
    $f_1 : M \to M$
    such that
    \begin{math}
        \varphi \inv \circ f_1^k \circ \varphi (x) = b x
          \end{math}
    for all $x  \in  [-1,1]^d$.
    Moreover,
    we may freely assume that $\varphi([-1,1]^d) \subof f_0(U)$
    and that
    $f_1(x) = f_0(x)$ for all $x  \in  M \sans f_0(U)$.
    By abuse of notation,
    we identify $[-1,1]^d$ with its image
    and regard $[-1,1]^d$ as a subset of $M$.
    
    Define a diffeomorphism $F_1$ of $M \ti \bbTD$ by
    $F_1(x,v) = (f_1(x,v)$, $Av + h_1(x))$.
    If $x  \in  \barU \sans f_1(U)$ and $v  \in  \bbTD$,
    define $\Eu_{F_1}(x,v) := \Eu_{F_0}(x,v)$.
    %Using \cite[theorem ???]{ham-dax},
    Using \cref{thm:genchaindom},
    one may then establish the existence of a dominated splitting
    $\Eu_{F_1} \oplus \Ecs_{F_1}$
    on all of $M \ti \bbTD$.
    Similarly,
    If $x  \in  \barU \sans f_1(U)$ and $v  \in  \bbTD$,
    define $\Ecu_{F_1}(x,v) := \Ecu_{F_0}(x,v)$
    and apply the same reasoning to establish a dominated splitting of
    the form
    $\Ecu_{F_1} \oplus \Es_{F_1}$
    on all of $M \ti \bbTD$.
    From this, one may show that $F_1$ is partially hyperbolic
    and that the splitting of $F_1$ is graph like at a point $x$
    if and only if the original $F_0$ was graph like at $x$.

    Since $\Eu_{F_1}$ is continuous and graph like on $U$,
    there is a uniform constant $C > 1$
    such that if $x  \in  [-1,1]^d \subof M$
    with $b  \le  \|x\|  \le  1$ and $v  \in  \bbTD$
    then 
    $\Eu_{F_1}(x,v)$
    is the graph of a linear function $L : \Eu_A(x,v) \to T_x M \oplus \Es_A(x,v)$
    with $\|L\| < C$.
    A similar bound also holds when $\Ecu_{F_1}(x,v)$ is expressed as the graph
    of a linear function.
    By \cref{cor:fullshearup},
    there are functions $f : M \to M$
    and $h : M \to \bbTD$
    such that $F$ defined by
    $F(x,v) = (f(x)$, $Av + h(x))$
    satisfies the following properties.
    \begin{enumerate}
        \item If either $x  \in  M \sans [-1,1]^d$
        or $x  \in  [-1,1]^d$ with $\|x\| > 1$,
        then $f(x) = f_1(x)$ and $h(x) = h_1(x)$.

        \item If $x  \in  [-1,1]^d$ with $b  \le  \|x\|  \le  1$,
        then $f(x) = b x$.
        Further,
        if $\{n_j\} \subof \bbN$ is such that
        $F^{n_j}(x,v)$ converges to a point
        $(x_0,v_0)$ in $x_0 \ti \bbTD$,
        then
        $DF^{n_j}(\Eu_{F_1}(x,v))$ converges to $0 \ti \Eu_A(v_0)$
        and
        $DF^{n_j}(\Ecu_{F_1}(x,v))$ converges to $0 \ti T_{v_0} \bbTD$.
    \end{enumerate}
    Then
    %\cite[theorem ???]{ham-dax}
    \cref{thm:genchaindom}
    shows that $F$ is partially hyperbolic
    with $x_0 \ti \bbTD$ tangent to $\Ecu_{F}$.
\end{proof}
With \cref{prop:addcu} established,
\cref{prop:allsinks} easily follows.

\begin{proof}[Proof of \cref{prop:allsinks}]
    Given $f_0$, define a hyperbolic toral automorphism $A : \bbTD \to \bbTD$
    such that $F_0 := f_0 \ti A$ is partially hyperbolic.
    For instance, $A$ can be the direct product of 
    $d$ copies of a high iterate of the cat map.
    Clearly, $F_0$ is a partially hyperbolic $A$-map
    and the splitting is graph like at all points.
    Let $x_0$ be any point in $X$ and apply \cref{prop:addcu}
    to $F_0$ and $x_0$
    to produce a map $F_1$ where $x_0 \ti \bbTD$ is tangent either to
    $\Ecs$ and $\Ecu$.
    If $X$ contains a point $x_1$ which is not in the orbit of $x_0$,
    then apply \cref{prop:addcu}
    to $F_1$ and $x_1$ to produce a map $F_2$.
    After a finite number of steps of this form,
    the desired map $F$ in \cref{prop:allsinks} is constructed.
\end{proof}

% Acknowledgements {{{1
\bigskip

\acknowledgement
The author thanks Christian Bonatti,
Andrey Gogolev,
Ni\-colas Gourmelon, Rafael Potrie,
and the Monash Topology Writing Group
for helpful discussions.

% Epilogue {{{1

%\bibliographystyle{plain}
\bibliographystyle{alpha}
\bibliography{dynamics}

\end{document}